\documentclass[12pt,a4paper]{amsart}
\usepackage{amsfonts}
\usepackage{amssymb}
\usepackage{amsmath}
\usepackage{amsthm}
\usepackage{cite}
\usepackage{graphicx}
\usepackage{amscd}
\usepackage{color}
\newtheorem{theorem}{Theorem}
\theoremstyle{plain}

\newtheorem{definition}[theorem]{Definition}

\newtheorem{corollary}[theorem]{Corollary}
\newtheorem{lemma}[theorem]{Lemma}
\newtheorem{remark}{Remark}

\newtheorem{proposition}[theorem]{Proposition}

\numberwithin{equation}{section}

\def\eps{\varepsilon}

\def\ri{{\rm i}}

\newcommand{\R}{\mathbb{R}}

\newcommand{\N}{\mathbb{N}}

\newcommand{\Z}{\mathbb{Z}}

\newcommand{\cD}{{\mathcal D}}

\renewcommand{\phi}{\varphi}

\DeclareMathOperator{\const}{const.}

\DeclareMathOperator{\kernel}{Ker}
\DeclareMathOperator{\range}{Rg}
\DeclareMathOperator{\spann}{span}
\DeclareMathOperator*{\essinf}{ess\,inf}
\DeclareMathOperator*{\esssup}{ess\,sup}
\DeclareMathOperator{\sign}{sign}

\newcommand{\dist}{\text{\rm dist}}
\newcommand{\supp}{\text{\rm supp}}

\def\blem{\begin{lemma}}\def\elem{\end{lemma}}
\def\bthm{\begin{theorem}}\def\ethm{\end{theorem}}
\def\bcor{\begin{corollary}}\def\ecor{\end{corollary}}
\def\beq{\begin{equation}}\def\eeq{\end{equation}}

\begin{document}

\title[Surface gap soliton ground states for the NLS]{Surface gap soliton ground states for the nonlinear Schr\"odinger equation} 

\author{Tom\'{a}\v{s} Dohnal}
\address{T. Dohnal \hfill\break 
Institut f\"{u}r Wissenschaftliches Rechnen und Mathematische Modellbildung, Karlsruhe Institute of Technology (KIT)\hfill\break
D-76128 Karlsruhe, Germany}
\email{dohnal@kit.edu}

\author{Michael Plum}
\address{M. Plum \hfill\break 
Institut f\"ur Analysis, Karlsruhe Institute of Technology (KIT)\hfill\break
D-76128 Karlsruhe, Germany}
\email{michael.plum@kit.edu}

\author{Wolfgang Reichel}
\address{W. Reichel \hfill\break 
Institut f\"ur Analysis, Karlsruhe Institute of Technology (KIT), \hfill\break
D-76128 Karlsruhe, Germany}
\email{wolfgang.reichel@kit.edu}

\date{\today}


\begin{abstract} 
We consider the nonlinear Schr\"{o}dinger equation
$$(-\Delta +V(x))u = \Gamma(x) |u|^{p-1}u, \quad x\in \R^n$$ 
with $V(x) = V_1(x), \Gamma(x)=\Gamma_1(x)$ for $x_1>0$ and $V(x) = V_2(x), \Gamma(x)=\Gamma_2(x)$ for $x_1<0$, where $V_1, V_2, \Gamma_1, \Gamma_2$ are periodic in each coordinate direction. This problem describes the interface of two periodic media, e.g. photonic crystals. We study the existence of ground state $H^1$ solutions (surface gap soliton ground states) for $0<\min \sigma(-\Delta +V)$. Using a concentration compactness argument, we provide an abstract criterion for the existence based on ground state energies of each periodic problem (with $V\equiv V_1, \Gamma\equiv \Gamma_1$ and $V\equiv V_2, \Gamma\equiv \Gamma_2$) as well as a more practical criterion based on ground states themselves. Examples of interfaces satisfying these criteria are provided. In 1D it is shown that, surprisingly, the criteria can be reduced to conditions on the linear Bloch waves of the operators $-\tfrac{d^2}{dx^2} +V_1(x)$ and $-\tfrac{d^2}{dx^2} +V_2(x)$. 
\end{abstract}

\maketitle


\section{Introduction}

The existence of localized solutions of the stationary nonlinear Schr\"{o}dinger equation (NLS)
\beq\label{E:NLS}
(-\Delta +V(x))u = \Gamma(x) |u|^{p-1}u, \quad x\in \R^n
\eeq
with a linear potential $V$ and/or a nonlinear potential $\Gamma$ is a classical problem of continued interest. The present paper deals with the existence of ground states in the case of two periodic media in $\R^n$ joined along a single interface, e.g. along the hyperplane $\{x_1=0\}\subset \R^n$. Both the coefficients $V$ and $\Gamma$ are then periodic on either side of the interface but not in $\R^n$. Exponentially localized solutions of this problem are commonly called surface gap solitons (SGSs) since they are generated by a surface/interface, and since zero necessarily lies in a gap (including the semi-infinite gap) of the essential spectrum of $L:=-\Delta + V$. Ground states in the case of purely periodic coefficients, where the solutions are refereed to as (spatial) gap solitons (GSs), were shown to exist in \cite{pankov} in all spectral gaps of $L$. The proof of \cite{pankov} does not directly apply to the interface problem due to the lack of periodicity in $\R^n$. Also, in contrast to the purely periodic case the operator $L$ in the interface problem can have eigenvalues \cite{doplurei,KOR05}. The corresponding eigenfunctions are localized near the interface so that it acts as a waveguide. In this paper we restrict our attention to ground states in the semiinfinite gap to the left of all possible eigenvalues, i.e. $0< \min \sigma(L)$. Using a concentration-compactness argument, we prove an abstract criterion ensuring ground state existence based on the energies of ground states of the two purely periodic problems on either side of the interface. We further provide a number of interface examples that satisfy this criterion. Moreover, in the case $n=1$ we give a sufficient condition for the criterion using solely linear Bloch waves of the two periodic problems.

\medskip

The physical interest in wave propagation along material interfaces stems mainly from the possibilities of waveguiding and localization at the interface. The problem with two periodic media is directly relevant in optics for an interface of two photonic crystals. Gap solitons in nonlinear photonic crystals are of interest as fundamental `modes' of the nonlinear dynamics but also in applications due to the vision of GSs being used in optical signal processing and computing \cite{Mingaleev02}. 

\medskip

The NLS \eqref{E:NLS} is a reduction of Maxwell's equations for monochromatic waves in photonic crystals with a $p$-th order nonlinear susceptibility $\chi^{(p)}$ when higher harmonics are neglected. In the following let $c$ be the speed of light in vacuum and $\eps_r$ the relative permittivity of the material. In 1D crystals, i.e., $\eps_r=\eps_r(x_1), \chi^{(p)}=\chi^{(p)}(x_1)$, equation \eqref{E:NLS} arises for the electric field ansatz $E(x,t)=(0,u(x_1),0)^T e^{\ri(kx_3 -\omega t)} + $c.c. and the potentials become $V(x_1) = k^2-\tfrac{\omega^2}{c^2}\eps_r(x_1), \Gamma(x_1)=\tfrac{\omega^2}{c^2} \chi^{(p)}(x_1)$. In 2D crystals, i.e., $\eps_r=\eps_r(x_1,x_2), \chi^{(p)}=\chi^{(p)}(x_1,x_2)$, the ansatz $E(x,t) = (0,0,u(x_1,x_2))^T e^{-\ri\omega t}+$c.c. leads to $V(x_1,x_2)=-\tfrac{\omega^2}{c^2}\eps_r(x_1,x_2)$ and $\Gamma(x_1,x_x)=\tfrac{\omega^2}{c^2}\chi^{(p)}(x_1,x_2)$. The physical condition $\eps_r>1$, valid for dielectrics, however, clearly limits the range of allowed potentials $V$. 

\medskip

On the other hand, the NLS is also widely used by physicists as an asymptotic model for slowly varying envelopes of wavepackets in 1D and 2D photonic crystals. In cubically nonlinear crystals the governing model is
\beq\label{E:NLS_dyn}
\ri \partial_{x_3}\phi + \Delta_\bot \phi + \tilde{V}(x_\bot) \phi + \Gamma(x_\bot)|\phi|^2\phi=0,
\eeq
where $x_\bot = x_1$ in 1D and $x_\bot = (x_1,x_2)$ in 2D, see e.g. \cite{Efremidis03,Kartashov06}. The ansatz $\phi(x) = e^{\ri kx_3} u(x_\bot)$ then leads to \eqref{E:NLS} with $V(x) = k-\tilde{V}(x)$. 

\medskip

GSs are also widely studied in Bose-Einstein condensates, where \eqref{E:NLS_dyn} is the governing equation for the condensate wave function without any approximation (with $x_3$ playing the role of time) \cite{Louis03}. It is referred to as the Gross-Pitaevskii equation and the periodic potential is typically generated by an external optical lattice.

\medskip

SGSs of the 1D and 2D NLS have been studied numerically in a variety of geometries and nonlinearities including case where only $V$ has an interface and $\Gamma$ is periodic in $\R^n$ \cite{Kartashov06,Makris06}, or vice versa \cite{do_pel,Blank_Dohnal} (1D), or where both $V$ and $\Gamma$ have an interface \cite{Kartashov08}.

\medskip

Optical SGSs have been also observed experimentally in a number of studies, as examples we list: SGSs at the edge of a 1D \cite{Rosberg06} and a 2D \cite{Wang07,Szameit07} photonic crystal as well as at the interface of two dissimilar crystals \cite{Suntsov08,Szameit08}.

\section{Mathematical Setup and Main Results}

Let $V_1, V_2, \Gamma_1, \Gamma_2: \R^n\to \R$ be bounded functions, $p>1$, and consider the differential operators
$$
L_i := -\Delta + V_i, \quad i=1,2
$$
on $\cD(L_i)=H^2(\R^n)\subset L^2(\R^n)$ and denote their spectra with $\sigma(L_i)$. Our basic assumptions are:
\begin{itemize}
\item[(H1)] $V_1, V_2, \Gamma_1, \Gamma_2$ are $T_k$-periodic in the $x_k$-direction for $k=1,\ldots,n$ with $T_1=1$,
\item[(H2)] $\esssup\Gamma_i>0$, $i=1,2$,
\item[(H3)] $1<p<2^\ast-1$,
\end{itemize}
where, as usual, $2^\ast=\frac{2n}{n-2}$ if $n\geq 3$ and $2^\ast=\infty$ if $n=1,2$. Let us also mention a stronger form of (H2), namely
\begin{itemize}
\item[(H2')] $\essinf \Gamma_i>0$ , $i=1,2$.
\end{itemize}
Conditions (H2), (H2') will be commented below. (H3) is commonly used in the variational description of ground states. In order to have an energy-functional $J$, which is well-defined on $H^1(\R^n)$, one needs $1\leq p\leq 2^\ast-1$. The assumption $p>1$ makes the problem superlinear and the assumption $p<2^\ast-1$ provides some compactness via the Sobolev embedding theorem. Although we restrict our attention to $1<p<2^\ast-1$, problem \eqref{E:NLS} for $0<p<1$ or for $p\geq 2^\ast-1$, $n\geq 3$ is also of interest. 

\medskip

Consider the two purely periodic (stationary) nonlinear Schr\"odinger equations
\begin{equation}
L_i u = \Gamma_i(x) |u|^{p-1}u \mbox{ in } \R^n \label{per_nls}.
\end{equation}
Their solutions arise as critical points of the functionals
$$
J_i[u] := \int_{\R^n} \left(\frac{1}{2}\left(|\nabla u|^2+ V_i(x) u^2\right)-\frac{1}{p+1} \Gamma_i(x) |u|^{p+1}\right)\,dx, \quad u \in H^1(\R^n).
$$
If (H1), (H2'), (H3) hold and if $0\not\in \sigma(L_i)$ then it is well known, cf. Pankov \cite{pankov}, that the purely periodic problem \eqref{per_nls} has a ground state $w_i$, i.e. a function $w_i\in H^1(\R^n)$ which is a weak solution of \eqref{per_nls} such that its energy $c_i := J_i[w_i]$ is minimal among all nontrivial $H^1$ solutions. However, under the additional assumption $0<\min \sigma(L_i)$ a ground state of \eqref{per_nls} exists under the weaker hypotheses (H1), (H2), (H3). This can be seen from an inspection of the proof in \cite{pankov}, which we leave to the reader. In the following, we use (H1), (H2), (H3). Of course our results are also valid under the stronger hypotheses (H1), (H2'), (H3), which do not require any extension of \cite{pankov}.

\medskip

In the present paper we are interested in ground states for a nonlinear Schr\"odinger equation modeling an interface between two different materials. For this purpose we define the composite functions 
$$
V(x) = \left\{
\begin{array}{ll}
V_1(x), & x\in \R^n_+, \vspace{\jot}\\
V_2(x), & x\in \R^n_-,
\end{array} \right.
\qquad
\Gamma(x) = \left\{
\begin{array}{ll}
\Gamma_1(x), & x\in \R^n_+, \vspace{\jot}\\
\Gamma_2(x), & x\in \R^n_-,
\end{array} \right.
$$
where $\R^n_\pm = \{x\in \R^n: \pm x_1>0 \}$ and the differential operator
$$
L := -\Delta + V \mbox{ on } \cD(L)=H^2(\R^n)\subset L^2(\R^n).
$$
We will prove existence of ground states for the nonlinear Schr\"odinger equation of interface-type
\begin{equation}
Lu = \Gamma(x) |u|^{p-1}u \mbox{ in } \R^n.
\label{int_nls}
\end{equation}
Solutions of \eqref{int_nls} are critical points of the energy functional
$$
J[u] := \int_{\R^n} \left(\frac{1}{2}\left(|\nabla u|^2+ V(x) u^2\right)-\frac{1}{p+1} \Gamma(x) |u|^{p+1}\right)\,dx, \quad u \in H^1(\R^n).
$$
Since $J$ is unbounded from above and below, minimization/maximization of $J$ on all of $H^1(\R^n)$ is impossible. Therefore we seek solutions of the following constrained minimization problem:
\begin{equation}
\mbox{find } w\in N \mbox{ such that } J[w] = c:= \inf_{u\in N} J[u], 
\label{int_gs}
\end{equation}
where $N$ is the Nehari manifold given by
$$ N=\{u\in H^1(\R^n): u\not =0, G[u]=0\}, \quad G[u] = \int_{\R^n} \left(|\nabla u|^2+ V(x)u^2- \Gamma(x)|u|^{p+1}\right) \,dx.
$$
Note that $N$ contains all non-trivial $H^1(\R^n)$-solutions of \eqref{int_nls} and (H2) ensures that $N\neq\emptyset$.\footnote{For a proof, construct a sequence $(u_k)_{k\in \N}$ in $H^1(\R^n)$ which converges in $L^{p+1}(\R^n)$ to the characteristic function of $\{x\in \R^n: \Gamma(x)>0, |x|\leq R\}$ for some large $R$. For sufficiently large $k\in \N$ one finds $\int_{\R^n} \Gamma(x) |u_k|^{p+1}\,dx>0$ and thus $\exists t>0$ such that $tu_k\in N$.} The stronger condition (H2') makes $N$ a topological sphere, i.e, $\forall u \in H^1(\R^n)\setminus\{0\}$ $\exists t>0$ such that $tu\in N$. Moreover, one of the advantages of $N$ is that the Lagrange multiplier introduced by the constraint turns out to be zero as checked by a direct calculation. In our case of a pure power nonlinearity one could alternatively use the constraint $\int_{\R^n}\Gamma(x)|u|^{p+1} \,dx = 1$. Here the Lagrange parameter can be scaled out a posteriori. A third possibility would be the constraint $\int_{\R^2}u^2\, dx=\mu>0$, which generates a Lagrange parameter $\lambda(\mu)$. The additional term $\lambda(\mu)u$ in \eqref{per_nls} cannot be scaled out. This approach is, moreover, restricted to $1<p<1+\tfrac{4}{n}$, cf. \cite{stuart}.

\begin{definition} 
The following terminology will be used throughout the paper.
\begin{enumerate}
\item[(a)] A \textit{bound state} is a weak solution of \eqref{per_nls} in $H^1(\R^n)$.
\item[(b)] A \textit{ground state} is a bound state such that its energy is minimal among all nontrivial bound states.
\item[(c)] A \textit{strong ground state} is a solution to \eqref{int_gs}.
\end{enumerate}
\end{definition}
Note that a strong ground state is a also a ground state because $N$ contains all non-trivial bound states.

\medskip


For the success in solving \eqref{int_gs} we need to assume additionally to (H1)--(H3) that $0 < \min \sigma(L)$, which is, e.g., satisfied if there exists a constant $v_0>0$ such that $V_1, V_2\geq v_0$. Note that $\sigma(L)\supset \sigma(L_1)\cup\sigma(L_2)$, and hence the assumption $0<\min\sigma(L)$ implies in particular $0<\min\sigma(L_i)$ for $i=1,2$. The additional spectrum of $L$ may be further essential spectrum or, as described in \cite{doplurei}, so-called gap-eigenvalues.

As we show in Lemma \ref{L:c_less_c12}, one always has $c\leq \min\{c_1,c_2\}.$ Our main result shows that if the strict inequality holds, then strong ground states exist.

\begin{theorem} Assume (H1)--(H3) and $0< \min \sigma(L)$. If $c<\min\{c_1,c_2\}$, where $c$ is defined in \eqref{int_gs} and $c_1, c_2$ are the ground state energies of the purely periodic problems \eqref{per_nls}, then $c$ is attained, i.e., there exists a strong ground state for the interface problem \eqref{int_nls}.
\label{main}
\end{theorem}

\begin{remark} We state the following two basic properties of every strong ground state $u_0$ of \eqref{int_nls}. 
\begin{itemize}
\item[(i)] $u_0$ is exponentially decaying. The proof given in \cite{pankov} can be applied.
\item[(ii)] Up to multiplication with $-1$, $u_0$ is strictly positive. For the reader's convenience a proof is sketched in Lemma~A2 of the Appendix.
\end{itemize}
\label{positivity}
\end{remark}

Let us also note a result which excludes the existence of strong ground states. 

\begin{theorem}
Assume (H1)--(H3) and $0< \min \sigma(L)$. If $V_1\leq V_2$ and $\Gamma_1\geq \Gamma_2$ and if one of the inequalities is strict on a set of positive measure, then there exists no strong ground state of \eqref{int_nls}.
\label{non_ex}
\end{theorem}

\begin{remark}\label{rem:curved}
The non-existence result in Theorem \ref{non_ex} can be extended to more general interfaces $\Sigma$ (not necessarily manifolds) as follows. Let $\Sigma$ separate $\R^n$ into two disjoint sets $\Omega_1$ and $\Omega_2$ with $\Omega_1$ unbounded such that $\R^n = \overline{\Omega}_1 \cup \overline{\Omega}_2$ and  $\Sigma = \partial \Omega_1 = \partial \Omega_2$ and suppose
\[V(x) = \left\{
\begin{array}{ll}
V_1(x), & x\in \Omega_1, \vspace{\jot}\\
V_2(x), & x\in \Omega_2,
\end{array} \right.
\qquad
\Gamma(x) = \left\{
\begin{array}{ll}
\Gamma_1(x), & x\in \Omega_1, \vspace{\jot}\\
\Gamma_2(x), & x\in \Omega_2.
\end{array} \right.\]
Then the previous non-existence result holds if there exists a sequence $(q_j)_{j\in \N}$ in $\Omega_1\cap {\mathcal Z}^{(n)}$ such that $\dist(q_j, \Omega_2) \rightarrow \infty$ as $j\rightarrow \infty$, where ${\mathcal Z}^{(n)}=T_1\Z\times T_2\Z\times\ldots\times T_n\Z$. 

This is, for instance, satisfied if there exists a cone ${\mathcal C}_1$ in $\R^n$ such that outside a sufficiently large ball $B_R$ the cone lies completely within the region $\Omega_1$, i.e., there is a sufficiently large radius $R>0$ such that
\[{\mathcal C}_1 \cap B_R^c \subset \Omega_1.\]
In the case $n=2$, where the cone becomes a sector, an example of such an interface is plotted in Figure~\ref{F:curved_interf}.
\begin{figure}[!ht]
  \begin{center}
  \scalebox{0.28}{\includegraphics{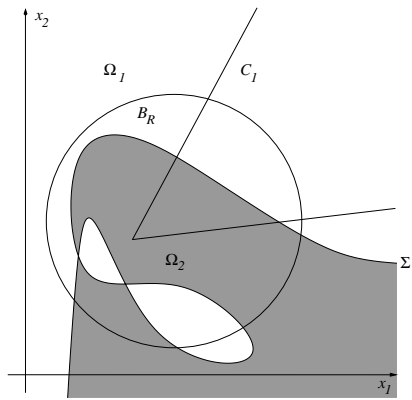}}
  \end{center}
  \caption{An example of a curved interface in 2D}
  \label{F:curved_interf}
\end{figure} 
\end{remark}

\medskip

\noindent
Note that (to the best of our knowledge) the fact that no strong ground state exists does not preclude the existence of a ground state. Moreover, under the assumptions of Theorem~\ref{non_ex} there may still exist bound states of \eqref{int_nls}, i.e. critical points of $J$ in $H^1(\R^n)$, cf. \cite{do_pel,Blank_Dohnal}.

\medskip

To verify the existence condition $c<\min\{c_1,c_2\}$ of Theorem~\ref{main},
it suffices to determine a function $u\in N$ such that $J[u]<\min\{c_1,c_2\}$. The following theorem shows that a suitable candidate for such a function is a shifted and rescaled ground state corresponding to the purely periodic problem with the smaller of the two energies. Using an idea of Arcoya, Cingolani and G\'{a}mez \cite{acg}, we shift the ground state far into the half-space representing the smaller ground state energy. The rescaling is needed to force the candidate to lie in $N$.

\medskip

The following theorem thus offers a criterion for verifying existence of strong ground states. 

\begin{theorem}
Assume (H1)--(H3) and $0<\min \sigma(L)$. Let $w_i$ be a ground state of \eqref{per_nls} for $i=1,2$ and let $e_1$ be the unit vector $(1, 0, \ldots, 0) \in \R^n$. 
\begin{itemize} 
\item[(a)] If $c_1\leq c_2$ and 
\beq\label{E:V_cond}
(p+1)\int_{\R^n_-} (V_2(x)-V_1(x)) w_1(x-te_1)^2\,dx \\
< 2 \int_{\R^n_-} (\Gamma_2(x)-\Gamma_1(x)) |w_1(x-te_1)|^{p+1}\,dx
\eeq
for all $t\in \N$ large enough, then $c<\min\{c_1,c_2\}$ and therefore \eqref{int_nls} has a strong ground state. Condition \eqref{E:V_cond} holds, e.g., if $\esssup (V_2-V_1)<0$.
\item[(b)] If $c_2\leq c_1$, then in the above criterion one needs to replace $w_1(x-te_1)$ by $w_2(x+te_1)$, integrate over $\R^n_+$ and switch the roles of $V_1, V_2$ and $\Gamma_1, \Gamma_2$. 
\end{itemize}
\label{gen_princ}
\end{theorem}

Using the criterion in Theorem \ref{gen_princ}, we have found several classes of interfaces leading to the existence of strong ground states. These are listed in the following theorems.

\medskip

The first example considers potentials related by a particular scaling.
\begin{theorem}
Assume (H1)--(H3), $0<\min \sigma(L)$, and $V_1(x) = k^2 V_2(k x)$, $\Gamma_1(x) = \gamma^2 \Gamma_2(k x)$ for some $k\in \N$ as well as 
\beq\label{E:V_scale_ord}
\sup_{[0,1]^n} V_2<k^2\inf_{[0,1]^n} V_2 \quad \mbox{ and } \quad k^\frac{n+2-p(n-2)}{p-1}\leq \gamma^\frac{4}{p-1}.
\eeq
Then \eqref{int_nls} has a strong ground state.
\label{T:scaled_interf}
\end{theorem}

The next theorem guarantees existence for interfaces with a large jump in $\Gamma$.


\begin{theorem}\label{T:large_jump}
Assume (H1)--(H3), $0<\min\sigma(L)$.
\begin{itemize}
\item[(a)] Let $\esssup (V_2-V_1)<0$. Then there exists a value $\beta_0>0$ depending on $c_2$ such that if $\Gamma_1(x)  \geq \beta_0$ almost everywhere, then \eqref{int_nls} has a strong ground state.
\item[(b)] A similar result holds for $\esssup (V_2-V_1)>0$ and $\Gamma_2(x)  \geq \beta_0$ with $\beta_0=\beta_0(c_1)$.
\end{itemize}
\end{theorem}

Finally, for the case $n=1$ we provide sufficient conditions for criterion \eqref{E:V_cond}. Instead of ground states $w_1, w_2$ themselves the new sufficient conditions use solely the linear Bloch modes of the operators $L_1, L_2$ on a single period. The coefficient $\Gamma$ in front of the nonlinear term has no influence in this criterion besides allowing the correct ordering between $c_1$ and $c_2$. Moreover, we show that if $0$ is sufficiently far from $\sigma(L)$, these conditions can be easily checked from the behavior of $V_1$ and $V_2$ near $x=0$.

\begin{theorem}
Let $n=1$ and consider the operator $L = -\frac{d^2}{dx^2} +V(x) -\lambda$ on ${\mathcal D}(L) = H^2(\R)\subset L^2(\R)$ with $\lambda \in \R$. Assume (H1)--(H3) and $0< \min \sigma(L)$. For $i=1,2$ define by $c_i$ the ground state energy of $(-\frac{d^2}{dx^2} +V_i(x) -\lambda)u=\Gamma_{i}(x)|u|^{p-1}u$ on $\R$ and assume $c_1\leq c_2$. 
\begin{itemize}
\item[(a)] A sufficient condition for the existence of a strong ground state for $Lu=\Gamma(x)|u|^{p-1}u$ in $\R$ is 
\begin{equation}
\int_{-1}^0 \Big(V_2(x)-V_1(x)\Big) p_-(x)^2 e^{2\kappa x}\,dx<0,
\label{gen_cond1}
\end{equation}
where $p_-(x)e^{\kappa x}$ is the Bloch mode decaying at $-\infty$ of $-\frac{d^2}{dx^2}+ V_1(x)-\lambda$.
\item[(b)] If for some $\eps>0$ the potentials $V_1, V_2$ are continuous in $[-\eps,0)$ and satisfy $V_2(x)<V_1(x)$ for all $x\in [-\eps,0)$, then condition \eqref{gen_cond1} holds for $\lambda$ sufficiently negative. In particular, if $V_1$, $V_2$ are $C^1$-functions near $x=0$, then 
\begin{equation}
V_2(0)<V_1(0) \qquad \mbox{ or } \qquad V_2(0)=V_1(0) \mbox{ and } V_2'(0)>V_1'(0)
\label{gen_cond2}
\end{equation}
implies \eqref{gen_cond1} for $\lambda$ sufficiently negative.
\label{T:Bloch_cond}
\end{itemize}
\end{theorem}

\noindent
\begin{remark} In the case $c_2\leq c_1$ the condition corresponding to \eqref{gen_cond1} is 
$$
\int_{0}^1 \Big(V_1(x)-V_2(x)\Big) p_+(x)^2 e^{-2\kappa x}\,dx<0,
$$
where $p_+(x)e^{-\kappa x}$ is the Bloch mode decaying at $+\infty$ of $-\frac{d^2}{dx^2}+ V_2(x)-\lambda$. It holds 
if $V_1, V_2$ satisfy the conditions of continuity and $V_1(x)<V_2(x)$ in $(0,\eps]$ for some $\eps>0$. The condition corresponding to \eqref{gen_cond2} is 
$$
V_1(0)<V_2(0) \qquad \mbox{ or } \qquad V_1(0)=V_2(0) \mbox{ and } V_2'(0)>V_1'(0).
$$
Note that the condition on the derivatives is the same as in \eqref{gen_cond2}.
\label{R:bloch_cond_reverse}
\end{remark}

In Section \ref{S:1D_example} we apply Theorem~\ref{T:Bloch_cond} to a so-called `dislocation' interface where $V_1(x) = V_0(x+\tau), V_2(x) = V_0(x-\tau)$ with $\tau\in \R$.

\medskip

The rest of the paper is structured as follows. In Section~\ref{sec:proofs}, using a concentration compactness argument, we prove the general criteria for existence/nonexistence of strong ground states in $\R^n$, i.e. Theorems~\ref{main}, \ref{non_ex}, and \ref{gen_princ}. Our two $n-$dimensional examples (Theorems~\ref{T:scaled_interf} and \ref{T:large_jump}), which satisfy these criteria, are proved in Section~\ref{S:nD_examples}. In Section~\ref{S:one-d} we prove Theorem~\ref{T:Bloch_cond}, i.e. a refinement of Theorem~\ref{gen_princ} for the case $n=1$. Section~\ref{S:1D_example} firstly presents a 1D example ($n=1$), namely a dislocation interface, satisfying the conditions of Theorem~\ref{T:Bloch_cond}, and secondly provides a heuristic explanation of the 1D existence results for $\lambda$ sufficiently negative. Finally, Section~\ref{S:discussion} discusses some open problems and the application of our results to several numerical and experimental works on surface gap solitons.

\section{$n$ Dimensions: General Existence Results} \label{sec:proofs}

\subsection{Proof of Theorem \ref{main}} \label{sec:ex}

According to Pankov \cite{pankov} ground states $w_i\in H^1(\R^n)$ for the purely periodic problem \eqref{per_nls} are strong ground states, so that they are characterized as 
$$
w_i\in N_i \quad \mbox{ such that } \quad J_i[w_i] = c_i := \inf_{u\in N_i} J_i[u]
$$
with $N_i$ being the Nehari-manifold
$$
N_i = \{u\in H^1(\R^n): u\not = 0, G_i[u]=0\}, \quad G_i[u] = \int_{\R^n} \left(|\nabla u|^2+ V_i(x)u^2- \Gamma_i(x) |u|^{p+1}\right) \,dx
$$
for $i=1,2$. The proof of Theorem \ref{main} consists in showing that a strategy similar to that of \cite{pankov} for purely periodic problems is successful in finding strong ground states of \eqref{int_nls} provided the basic inequality $c<\min\{c_1,c_2\}$ for the corresponding minimal energy levels holds.

\medskip

In the following we use the scalar product $\langle \phi,\psi\rangle = \int_{\R^n} \nabla \phi \cdot \nabla \psi+ \phi\psi\,dx$ for $\phi,\psi\in H^1(\R^n)$. For $u\in H^1(\R^n)$ the bounded linear functional $J'[u]$ can be represented by its gradient denoted by $\nabla J[u]$, i.e. 
$$
J'[u]\phi = \langle \nabla J[u],\phi\rangle \mbox{ for all } \phi\in H^1(\R^n).
$$

\begin{lemma} There exists a sequence $(u_k)_{k\in \N}$ on the Nehari-manifold $N$ such that $J[u_k]\to c$ and $J'[u_k]\to 0$ as $k\to \infty$. Moreover, $(u_k)_{k\in \N}$ is bounded and bounded away from zero in $H^1(\R^n)$, i.e., there exists $\epsilon, K>0$ such that $\epsilon\leq \|u_k\|_{H^1(\R^n)}\leq K$ for all $k\in \N$.
\label{minimizing_sequence}
\end{lemma}

\begin{proof} Let $|||u|||^2 = \int_{\R^n} |\nabla u|^2+ V(x) u^2\,dx$. Clearly $|||\cdot|||$ is equivalent to the standard norm on $H^1(\R^n)$ since $\min\sigma(L)>0$. Note that for $u\in N$ one finds $J[u]=\eta |||u|||^2$ with $\eta = \frac{1}{2}-\frac{1}{p+1}$. This explains why every minimizing sequence of $J$ on $N$ has to be bounded. Moreover every element $u\in N$ satisfies $|||u|||^2 = \int_{\R^n} \Gamma(x) |u|^{p+1}\,dx 
\leq C \|u\|_{H^1(\R^n)}^{p+1} \leq \bar C |||u|||^{p+1}$, and since $u\not = 0$, the lower bound on $u$ follows.

Now consider a sequence $0<\epsilon_k\to 0$ and a minimizing sequence $(v_k)_{k\in \N}$ of $J$ on $N$ such that $J[v_k]\leq c +\epsilon_k^2$. By using Ekeland's variational principle, cf. Struwe \cite{struwe}, there exists a second minimizing sequence $(u_k)_{k\in N}$ in $N$ such that $J[u_k]\leq J[v_k]$ and  
$$
J[u_k] < J[u] + \epsilon_k \|u-u_k\|_{H^1(\R^n)} \mbox{ for all } u\in N, \quad u \not = u_k.
$$
Consider the splitting $\nabla J[u_k] = s_k+t_k$ with $s_k \in (\kernel G'[u_k])^\perp$ and $t_k \in \kernel G'[u_k]$.
Due to the following Lemma \ref{approx_lagrange} we know that $\|t_k\|_{H^1(\R^n)}\leq \epsilon_k$. 
Note that the range $\range G'[u_k]=\R$ because $G'[u_k]u_k = (1-p)|||u_k|||^2 \not = 0$. Furthermore, $\spann(\nabla G[u_k])= (\kernel G'[u_k])^\perp$. Hence there exist real numbers $\sigma_k\in \R$ such that 
$s_k = -\sigma_k \nabla G[u_k]$, i.e.,
$$
\underbrace{\langle \nabla J[u_k],u_k\rangle}_{=0} + \sigma_k\langle \nabla G[u_k],u_k\rangle = \langle t_k,u_k\rangle.
$$
Thus,
$$
|\sigma_k| (p-1)|||u_k|||^2 \leq \epsilon_k \|u_k\|_{H^1(\R^n)},
$$
which shows that $\sigma_k\to 0$ as $k\to \infty$ since $|||u_k|||$ is bounded away from zero. Hence we have proved that $\nabla J[u_k]=s_k+t_k\to 0$ as $k\to \infty$. 
\end{proof}

\begin{lemma} Suppose $\epsilon>0$ and $u_0 \in N$ are such that $J[u_0]-J[u] \leq \epsilon\|u_0-u\|_{H^1(\R^n)}$ for all $u\in N$. If $\nabla J[u_0] = s+t$ with $s\in (\kernel G'[u_0])^\perp$ and $t \in \kernel G'[u_0]$, then $\|t\|_{H^1(\R^n)}\leq\epsilon$. 
\label{approx_lagrange}
\end{lemma}

\begin{proof} We split $u\in H^1(\R^n)$ such that $u=\tau u_0+v$ with $v\in \spann(u_0)^\perp$. The Fr\'{e}chet derivative $G'[u_0]$ may be split into the partial Fr\'{e}chet derivatives $\partial_1 G[u_0]:= G'[u_0]|_{\spann(u_0)}$, $\partial_2 G[u_0]:= G'[u_0]|_{\spann(u_0)^\perp}$, so that 
\begin{equation}
G'[u_0](\tau u_0+v) = \tau \partial_1 G[u_0] u_0 + \partial_2 G[u_0] v.
\label{split}
\end{equation}
Since $G'[u_0]u_0=(1-p)|||u_0|||^2\not =0$, we have that $\partial_1 G[u_0]$ is bijective, and hence by the implicit function theorem there exists a ball $B(0)\subset \spann(u_0)^\perp$ and a $C^1$-function $\tau: B(0) \to \R$ such that 
$$
G(\tau(v)u_0+v)=0, \qquad \tau(0)=1
$$
and 
\begin{equation}
(\tau'(0)v) u_0 = -(\partial_1 G[u_0])^{-1} \partial_2 G[u_0] v \mbox{ for all } v \in \spann(u_0)^\perp.
\label{ift}
\end{equation}
Define the linear map $\phi: \spann(u_0)^\perp\to H^1(\R^n)$ by 
$$
\phi(v) := (\tau'(0)v)u_0+v, \quad v\in \spann(u_0)^\perp.
$$
We claim that $\phi$ is a bijection between $\spann(u_0)^\perp$ and $\kernel(G'[u_0])$. First note that indeed $\phi$ maps into $\kernel(G'[u_0])$, which can be seen from \eqref{ift}. Let us prove that $\phi$ is injective: if $\phi(v)=(\tau'(0)v)u_0+v=0$, then clearly $v=0$. To see that $\phi$ is surjective, take $u\in\kernel(G'[u_0])$ and write $u=\theta u_0+v$ for some $\theta\in\R$ and some $v\in \spann(u_0)^\perp$. Then, by \eqref{split} and \eqref{ift}
$$
\theta u_0 = -(\partial_1 G[u_0])^{-1} \partial_2 G[u_0] v = (\tau'(0)v)u_0.
$$
Hence $u = \phi(v)$, and we have proved the bijectivity of the map $\phi$. Next, we compute for $u\in N$ near $u_0$, where $u=\tau(v)u_0+v$ with $v\in B(0)\subset \spann(u_0)^\perp$, that
$$
u-u_0 = (\tau(v)-1)u_0+v = (\tau'(0)v)u_0 + v + o(v)= \phi(v)+o(v) \mbox{ as } v\to 0.
$$
Therefore
\begin{equation}
J[u_0]-J[u] = J'[u_0](u_0-u) + o(u-u_0) = -J'[u_0]\phi(v)+o(v)
\label{J_a}
\end{equation}
and 
\begin{equation}
J[u_0]-J[u] \leq \epsilon\|u_0-u\|_{H^1(\R^n)} = \epsilon\|\phi(v)\|_{H^1(\R^n)}+o(v)
\label{J_b}
\end{equation}
by assumption. Setting $v=\lambda \bar v$ with $\bar v\in \spann(u_0)^\perp$ and letting $\lambda\to 0+$, we obtain from \eqref{J_a}, \eqref{J_b}
$$
-J'[u_0]\phi(\bar v) \leq \epsilon \|\phi(\bar v)\|_{H^1(\R^n)} \mbox{ for all } \bar v \in \spann(u_0)^\perp.
$$
Due to the bijectivity of $\phi$ and by considering both $\bar v$ and $-\bar v$ we get
$$
|J'[u_0]w| \leq \epsilon \|w\|_{H^1(\R^n)} \mbox{ for all } w \in \kernel(G'[u_0]),
$$
which implies the claim. 
\end{proof}

Theorem \ref{main} will follow almost immediately from the next result.

\begin{proposition}
For $\delta>0$ let $S_\delta = (-\delta,\delta)\times \R^{n-1}\subset \R^n$ denote a strip of width $2\delta$. If there exists $\delta>0$ such that for the sequence of Lemma~\ref{minimizing_sequence} one has $\liminf_{k\in \N} \|u_k\|_{H^1(S_\delta)}=0$, then $c\geq \min\{c_1,c_2\}$.
\label{prop_result}
\end{proposition}

Proposition \ref{prop_result} will be proved via some intermediate results. We define a standard one-dimensional $C^\infty$ cut-off function such that
$$
\chi_\delta(t) = 1 \mbox{ for } t\geq \delta, \quad \chi_\delta(t)=0 \mbox{ for } t\leq 0
$$
and $0\leq \chi_\delta\leq 1$, $\chi_\delta'\geq 0$. From $\chi_\delta$ we obtain further cut-off functions
$$
\chi_\delta^+(x) := \chi_\delta(x_1), \quad \chi_\delta^-(x) := \chi_\delta(-x_1), \quad x=(x_1,\ldots,x_n)\in \R^n.
$$
Note that $\chi_\delta^\pm$ is supported in the half-space $\R^n_\pm$ and $\nabla\chi_\delta^\pm$ is supported in the strip $S_\delta$. 

\begin{lemma} 
Let $(u_k)_{k\in \N}$ be a bounded sequence in $H^1(\R^n)$ such that $\|u_k\|_{H^1(S_\delta)}\to 0$ as $k\to \infty$ and define
$$
v_k(x) := u_k(x)\chi_\delta^+(x), \quad w_k(x) := u_k(x)\chi_\delta^-(x).
$$
Then 
\begin{itemize}
\item[(i)] $\|u_k\|^2 = \|v_k\|^2+\|w_k\|^2+ o(1)$ where $\|\cdot\|$ can be the $L^2$-norm or the $H^1$-norm on $\R^n$,
\item[(ii)] $J[u_k] = J_1[v_k]+J_2[w_k]+o(1)$,
\item[(iii)] $J'[u_k] = J_1'[v_k]+J_2'[w_k]+o(1)$,
\end{itemize}
where $o(1)$ denotes terms converging to $0$ as $k\to \infty$ and convergence in (iii) is understood in the sense of $H^{-1}(\R^n):=(H^1(\R^n))^\ast$.
\label{properties}
\end{lemma}

\begin{proof} (i): First note that
\begin{eqnarray*}
 u_k^2 &=& u_k^2(\chi_\delta^++\chi_\delta^-)^2+u_k^2(1-\chi_\delta^+-\chi_\delta^-)^2+
2u_k^2(\chi_\delta^++\chi_\delta^-)(1-\chi_\delta^+-\chi_\delta^-)\\
& = & v_k^2+w_k^2+u_k^2(1-\chi_\delta^+-\chi_\delta^-)^2+2u_k^2(\chi_\delta^++\chi_\delta^-)(1-\chi_\delta^+-\chi_\delta^-).
\end{eqnarray*}
Furthermore, since
\begin{equation}
\nabla u_k = \nabla v_k + \nabla w_k + (1-\chi_\delta^+-\chi_\delta^-)\nabla u_k -u_k\nabla(\chi_\delta^++\chi_\delta^-),
\label{gradient}
\end{equation}
we find
\begin{align*}
|\nabla u_k|^2 =& |\nabla v_k|^2+|\nabla w_k|^2+(1-\chi_\delta^+-\chi_\delta^-)^2 |\nabla u_k|^2+ u_k^2(|\nabla \chi_\delta^+|^2+|\nabla \chi_\delta^-|^2)\\
& +2(1-\chi_\delta^+-\chi_\delta^-)(\nabla v_k+\nabla w_k)\cdot \nabla u_k-2u_k(\nabla v_k+\nabla w_k)\cdot \nabla (\chi_\delta^++\chi_\delta^-)\\
& -2u_k(1-\chi_\delta^+-\chi_\delta^-)\nabla u_k\cdot\nabla(\chi_\delta^++\chi_\delta^-).
\end{align*}
Integrating these expressions over $\R^n$ and observing that terms involving $(1-\chi_\delta^+-\chi_\delta^-)$ or $\nabla(\chi_\delta^++\chi_\delta^-)$ are supported in $S_\delta$, where $\|u_k\|_{H^1(S_\delta)}$ tends to zero, and that the sequences $(v_k)_{k\in\N}, (w_k)_{k\in \N}$ are bounded in $H^1(\R^n)$, we obtain the claim (i).

\smallskip

\noindent
(ii): Let us compute
$$
\int_{\R^n} \Gamma(x) |u_k|^{p+1}\,dx = \int_{\R^n_-} \Gamma_2(x) |w_k|^{p+1}\,dx + \int_{\R^n_+} \Gamma_1(x) |v_k|^{p+1}\,dx + I,
$$
where 
$$ I=\int_{S_\delta}\Gamma(x)\left(|u_k|^{p+1} -|w_k|^{p+1}-|v_k|^{p+1} \right)\,dx.
$$
By the assumption that $\|u_k\|_{H^1(S_\delta)}$ tends to zero and by the Sobolev-embedding theorem $I$ converges to $0$ as $k\to\infty$. A similar computation shows 
$$
\int_{\R^n} V(x) u_k^2\,dx = \int_{\R^n_-} V_2(x) w_k^2\,dx + \int_{\R^n_+} V_1(x) v_k^2\,dx + o(1) \mbox{ as } k\to \infty.
$$ 
Together with (i) we get the claim in (ii).

\smallskip

\noindent
(iii): Using \eqref{gradient}, we obtain
\begin{align*}
J'[u_k]&\phi = J'[v_k]\phi + J'[w_k]\phi \\
&+\int_{S_\delta} \Big((\nabla u_k \cdot \nabla \phi+V(x)u_k\phi)(1-\chi_\delta^+(x)-\chi_\delta^-(x)) - u_k\nabla\phi\cdot (\nabla \chi_\delta^+(x)+\nabla \chi_\delta^-(x))\Big)\,dx\\
&+\tilde{I}_k(\phi),
\end{align*}
where 
$$
\tilde{I}_k(\phi)=  \int_{S_\delta} \Gamma(x)\left(|u_k|^{p-1}u_k-|v_k|^{p-1}v_k-|w_k|^{p-1}w_k\right)\phi\,dx,
$$
and hence $\tilde{I}_k$ tends to $0$ in $H^{-1}(\R^n)$ as $k\to \infty$. Thus
\begin{eqnarray*}
J'[u_k]\phi= J_1'[v_k]\phi + J_2'[w_k]\phi + \tilde{I}_k(\phi) + \int_{S_\delta} \big((a_1(x)\nabla u_k+ a_2(x) u_k\big)\cdot \nabla \phi +a_3(x)u_k\phi\,dx,
\end{eqnarray*}
where the functions $a_1,\ldots,a_3$ are bounded on $S_\delta$. Using $\tilde{I}_k\to 0$ in $H^{-1}(\R^n)$ and once more that $u_k\to 0$ in $H^1(S_\delta)$ as $k\to \infty$ we obtain the claim of (iii).
\end{proof}

In order to proceed with the proof of Proposition~\ref{prop_result}, we quote the following famous concen\-tration-compactness result, cf. Lions \cite{Lions84}. With a minor adaptation of the proof given in Willem \cite{willem} one can state the following version. 

\begin{lemma}[P.L.Lions, 1984] For $0<a\leq \infty$ let $S_a=(-a,a)\times \R^{n-1}$. Let $0<r<a$, $s_0\in [2,2^\ast)$ and assume that $(u_k)_{k\in\N}$ is a bounded sequence in $H^1(S_{2a})$ such that 
$$
\lim_{k\to \infty} \sup_{\xi\in S_a} \left(\int_{B_r(\xi)} |u_k|^{s_0}\,dx\right)=0.
$$
Then $u_k\to 0$ as $k\to \infty$ in $L^s(S_a)$ for all $s\in (2,2^\ast)$.
\label{cc}
\end{lemma}

\noindent
{\em Proof of Proposition \ref{prop_result}:} By assumption we may select a subsequence (again denoted by $(u_k)$) from the sequence of Lemma~\ref{minimizing_sequence} such that $\lim_{k\rightarrow \infty} \|u_k\|_{H^1(S_\delta)}=0$. Recall that $u_k\in N$ satisfies
$$
|||u_k|||^2 = \int_{\R^n} \Gamma(x)|u_k|^{p+1}\,dx
$$
and that $|||u_k|||$ is bounded and bounded away from $0$ by Lemma~\ref{minimizing_sequence}. Hence no subsequence of  $\|u_k\|_{L^{p+1}(\R^n)}$ converges to $0$ as $k\to \infty$. By the concentration-compactness result of Lemma~\ref{cc} with $a=\infty$ we have that for any $r>0$ there exists $\epsilon>0$ such that 
$$
\liminf_{k\in \N} \sup_{\xi\in \R^n} \int_{B_r(\xi)} u_k^2 \,dx \geq 2\epsilon,
$$
and hence that there exists a subsequence of $u_k$ (again denoted by $u_k$) and points $\xi_k\in \R^n$ such that 
\begin{equation}
\int_{B_r(\xi_k)} u_k^2\,dx \geq \epsilon \mbox{ for all } k \in \N. 
\label{cc1}
\end{equation}
Next we choose vectors $z_k \in {\mathcal Z}^{(n)}=T_1\Z\times T_2\Z\times\ldots\times T_n\Z$ such that $(z_k-\xi_k)_{k\in \N}$ is bounded (recall that $T_1=1, T_2,\ldots,T_n>0$ denote the periodicities of the functions $V_1,V_2,\Gamma_1,\Gamma_2$ in the coordinate directions $x_1,\ldots,x_n$). Then there exists a radius $\rho\geq r+\sup_{k\in\N}|z_k-\xi_k|$ such that 
\begin{equation}
\int_{B_\rho(z_k)} u_k^2\,dx \geq \epsilon \mbox{ for all } k \in \N. 
\label{cca}
\end{equation}
We show that $(z_k)_{k\in \N}$ is unbounded in the $x_1$-direction. Assume the contrary and define for $x\in \R^n$
$$
\bar u_k(x) := u_k(x+z_k'), \quad z_k' = (0,(z_k)_2,\ldots,(z_k)_n).
$$
By the boundedness of $(z_k)_1$ there exists a radius $R\geq \rho$ such that 
\begin{equation}
\int_{B_R(0)} \bar u_k^2\,dx \geq \epsilon \mbox{ for all } k \in \N. 
\label{cca_shift}
\end{equation}
By taking a weakly convergent subsequence $\bar u_k\rightharpoonup \bar u_0$ in $H^1(\R^n)$ and using the compactness of the embedding $H^1(B_R(0))\to L^2(B_R(0))$, we have $\bar u_0\not = 0$. Moreover, if $\phi\in C_0^\infty(\R^n)$ and if we set $\phi_k(x):= \phi(x-z'_k)$, then we can use the periodicity of $V, \Gamma$ in the directions $x_2,\ldots,x_n$ to see that
$$
o(1) = J'[u_k]\phi_k = J'[\bar u_k]\phi \mbox{ as } k\to \infty,
$$ 
where the first equality is a property of the sequence $(u_k)_{k\in\N}$ as stated in Lemma~\ref{minimizing_sequence}. On one hand, $\int_{\R^n} \nabla \bar u_k \cdot \nabla \phi + V(x)\bar u_k\phi\,dx \to \int_{\R^n} \nabla \bar u_0\cdot \nabla \phi+ V(x) \bar u_0 \phi\,dx$ by the weak convergence of the sequence $\bar u_k\rightharpoonup \bar{u}_0$. On the other hand, by the compact Sobolev embedding $H^1(K)\to L^{p+1}(K)$ with $K=\supp\phi$ and the continuity of the Nemytskii operator $u \mapsto |u|^{p-1}u$ as a map from $L^{p+1}(K)$ to $L^\frac{p+1}{p}(K)$, cf. Renardy-Rogers \cite{RR}, we find that $\int_{\R^n} \Gamma(x) |\bar u_k|^{p-1}\bar u_k \phi\,dx \to \int_{\R^n} \Gamma(x) |\bar u_0|^{p-1} \bar u_0\phi\,dx$. Hence we have verified that the limit function $\bar u_0$ is a weak solution of $L \bar u_0=\Gamma(x) |\bar u_0|^{p-1}\bar u_0$ in $\R^n$. Standard elliptic regularity implies that $\bar u_0$ is a strong $W_{loc}^{2,q}(\R^n)$-solution for any $q\geq 1$. Since we also know that $\bar u_0\equiv 0$ on $S_\delta$, we can apply the unique continuation theorem, cf. Schechter, Simon \cite{sche_si} or Amrein et al. \cite{am_ber_ge}, to find the contradiction $\bar u_0\equiv 0$ in $\R^n$. Thus, $(z_k)_{k\in \N}$ is indeed unbounded in the $x_1$-direction.

\smallskip

Let $v_k, w_k$ be defined as in Lemma~\ref{properties} and define 
$$
\bar v_k(x) := v_k(x+z_k), \quad \bar w_k(x) := w_k(x+z_k), \quad x\in \R^n
$$
and observe that both $\bar v_k$ and $\bar w_k$ are bounded sequences in $H^1(\R^n)$. Moreover, for almost all $k$ we have 
$$ 
\|\bar v_k\|_{L^2(B_R(0))} \geq \frac{\sqrt{\epsilon}}{3} \mbox{ or } 
\|\bar w_k\|_{L^2(B_R(0))} \geq \frac{\sqrt{\epsilon}}{3}
$$
by Lemma~\ref{properties}(i) and \eqref{cca_shift}. Taking weakly convergent subsequences, we get that $\bar v_k\rightharpoonup \bar v_0$ and $\bar w_k\rightharpoonup \bar w_0$ where $\bar v_0\not = 0$ or $\bar w_0 \not = 0$. 
Since $z_k$ is unbounded in the $x_1$-direction, we may assume that either $(z_k)_1\to +\infty$ or $(z_k)_1\to -\infty$ as $k\to \infty$. In the first case $\bar w_k\rightharpoonup 0$ while in the second case $\bar v_k\rightharpoonup 0$ as $k\to\infty$. In the following, let us consider only the case $(z_k)_1\to +\infty$. Then, from Lemma~\ref{properties} and the periodicity of $V_1,V_2,\Gamma_1,\Gamma_2$ we have for any bounded sequence $\phi_k\in H^1(\R^n)$ that 
\begin{align}
o(1) &= J'[u_k]\phi_k=J_1'[v_k]\phi_k+J_2'[w_k]\phi_k+o(1) \label{special}\\
&=J_1'[\bar v_k]\phi_k(\cdot+z_k)+J_2'[\bar w_k]\phi_k(\cdot+z_k) +o(1) \mbox{ as } k\to\infty. \nonumber
\end{align}
If we apply \eqref{special} to $\phi_k(x) := \phi(x-z_k)$, where $\phi\in C_0^\infty(\R^n)$, then 
$$
o(1) = J_1'[\bar v_k]\phi+J_2'[\bar w_k]\phi+o(1) = J_1'[\bar v_k]\phi+o(1)
$$
since $\bar w_k\rightharpoonup 0$ as $k\to\infty$ (where we have again used the compact Sobolev embedding and the continuity of the Nemytskii operator). From this we can deduce that $\bar v_0$ is a nontrivial solution of 
\begin{equation}
L_1 \bar v_0= \Gamma_1(x) |\bar v_0|^{p-1}\bar v_0 \mbox{ in } \R^n.
\label{nls_v0}
\end{equation}
Applying \eqref{special} with $\phi_k = u_k$, one obtains
\begin{align*}
o(1) & = J'[u_k]u_k = J_1'[v_k]u_k + J_2'[w_k]u_k + o(1) \\
& = J_1'[v_k]v_k + J_2'[w_k]w_k + o(1) \\
& = J_1'[\bar v_k]\bar v_k + J_2'[\bar w_k]\bar w_k + o(1),
\end{align*}
which together with \eqref{nls_v0} implies
\begin{eqnarray}
\lefteqn{\liminf_{k\in \N} (J_1[\bar v_k]+J_2[\bar w_k])} \nonumber\\
&= &\liminf_{k\in\N} \Big(J_1[\bar v_k]+J_2[\bar w_k]-\frac{1}{p+1} (\underbrace{J_1'[\bar v_k]\bar v_k+J_2'[\bar w_k]\bar w_k}_{=o(1) \mbox{ as } k\to\infty})\Big)\nonumber\\
&= & \liminf_{k\in \N} \left(\frac{1}{2}-\frac{1}{p+1}\right)\int_{\R^n}
  \left(|\nabla \bar v_k|^2+V_1(x)\bar v_k^2+|\nabla \bar w_k|^2+V_2(x)\bar w_k^2\right)\,dx \label{split_J}\\
& \geq & \left(\frac{1}{2}-\frac{1}{p+1}\right)\int_{\R^n} |\nabla \bar v_0|^2+V_1(x)\bar v_0^2\,dx \nonumber\\
 &= & J_1[\bar v_0]. \nonumber
\end{eqnarray}
Lemma~\ref{properties}(ii) also implies
$$
J[u_k] = J_1[v_k]+J_2[w_k]+o(1)=J_1[\bar v_k]+J_2[\bar w_k]+o(1) \mbox{ as } k\to\infty,
$$
which together with \eqref{split_J} yields the result
$$
c= \lim_{k\to\infty} J[u_k] \geq J_1[\bar v_0] \geq c_1.
$$
In the case where $(z_k)_1\to -\infty$ we would have obtained $c=\lim_{k\to \infty} J[u_k]\geq c_2$. Hence, in any case we find $c\geq \min\{c_1,c_2\}$, which finishes the proof of Proposition~\ref{prop_result}. \qed

\medskip

\noindent
{\em Proof of Theorem~\ref{main}:} As in Lemma~\ref{minimizing_sequence} let $(u_k)_{k\in \N}$ be a minimizing sequence of $J$ on the Nehari-manifold $N$ such that $J'[u_k]\to 0$ as $k\to\infty$. From Proposition~\ref{prop_result} we know that $\liminf_{k\in\N} \|u_k\|_{H^1(S_\delta)}>0$ for any $\delta>0$. Let us fix $\delta>0$. By the following Lemma~\ref{neues_lemma} we know that for $0<R<2\delta$
\begin{equation}
\liminf_{k\in\N}\sup_{\xi\in S_{2\delta}} \left(\int_{B_R(\xi)} |u_k|^2\,dx\right)>0.
\label{true}
\end{equation}
Thus, there exist centers $\xi_k\in S_{2\delta}$ and $\epsilon>0$ such that 
$$
\int_{B_R(\xi_k)} u_k^2\,dx \geq \epsilon \mbox{ for all } k \in \N,
$$
and by choosing suitable vectors $z_k\in {\mathcal Z}^{(n)}$ with $(z_k)_1=0$ and a radius $\rho\geq R$, we may achieve that 
$$
\int_{B_\rho(z_k)} u_k^2\,dx \geq \epsilon \mbox{ for all } k \in \N.
$$
Setting 
$$
\bar u_k(x) := u_k(x+z_k), \quad x\in \R^n,
$$
we find 
$$
\int_{B_\rho(0)} \bar u_k^2\,dx \geq \epsilon \mbox{ for all } k \in \N. 
$$
Taking a weakly convergent subsequence $\bar u_k\rightharpoonup \bar u_0$ in $H^1(\R^n)$, we obtain by the argument given in the proof of Proposition \ref{prop_result} that $\bar u_0$ is a non-trivial weak solution of $L \bar u_0=\Gamma(x) |\bar u_0|^{p-1}\bar u_0$ in $\R^n$. Finally, as seen before in the proof of Proposition~\ref{prop_result}, one obtains 
\begin{align*}
c = \lim_{k\to\infty} J[u_k] =& \lim_{k\to\infty} \left(J[u_k]-\frac{1}{p+1}J'[u_k]u_k\right) \\
=& \lim_{k\to\infty} \left(J[\bar u_k]-\frac{1}{p+1}J'[\bar u_k]\bar u_k\right) \\
=& \lim_{k\to\infty} \left(\frac{1}{2}-\frac{1}{p+1}\right)\int_{\R^n}
  \left(|\nabla \bar u_k|^2+V(x)\bar u_k^2\right)\,dx \\
\geq & \left(\frac{1}{2}-\frac{1}{p+1}\right)\int_{\R^n}
  \left(|\nabla \bar u_0|^2+V(x)\bar u_0^2\right)\,dx\\
=& J[\bar u_0].
\end{align*}
Since $\bar u_0$ is non-trivial, it belongs to the Nehari manifold $N$. Thus, equality holds in the last inequality and $\bar u_0$ is a strong ground state.\qed

\begin{lemma} With the notation of the proof of Theorem~\ref{main}
$$
\liminf_{k\in\N}\sup_{\xi\in S_{2\delta}} \left(\int_{B_R(\xi)} |u_k|^2\,dx\right)>0.
\leqno \eqref{true}
$$\
\label{neues_lemma}
\end{lemma}

\begin{proof}
Otherwise, by concentration-compactness Lemma~\ref{cc} with $a=2\delta$ we find a subsequence such that $\|u_k\|_{L^s(S_{2\delta})}\to 0$ as $k\to \infty$ for all $s\in (2,2^\ast)$. This is impossible as the following argument shows. Since $J'[u_k]\to 0$ in $(H^1(\R^n))^\ast$ and $0\notin \sigma(L)$, there exists a sequence $(\zeta_k)_{k\in \N}$ in $H^1(\R^n)$ such that $L\zeta_k = J'[u_k]$ in $\R^n$ and $\|\zeta_k\|_{H^1(\R^n)}\to 0$ as $k\to \infty$. In particular $\theta_k := u_k-\zeta_k$ is a weak solution of 
\beq
L\theta_k = \Gamma(x) |u_k|^{p-1}u_k \mbox{ in } \R^n
\label{eq_zk}
\eeq
and
\beq 
\|u_k\|_{L^s(S_{2\delta})}, \|\theta_k\|_{L^s(S_{2\delta})} \to 0 \mbox{ as } k \to \infty \mbox{ for all } s\in (2,2^\ast).
\label{ls_norm}
\eeq 
As $0\not \in \sigma(L)$, we may use the fact $L^{-1}: L^q(\R^n)\to W^{2,q}(\R^n)$ is a bounded linear operator for all $q\in (1,\infty)$. This fact may be well known; for the reader's convenience we have given the details in the appendix, cf. Lemma A1. Thus, due to the boundedness of $(u_k)_{k\in \N}$ in $H^1(\R^n)$, 
\beq
\|\theta_k\|_{W^{2,q}(\R^n)} \leq \const \|u_k^p\|_{L^q(\R^n)} = O(1) \mbox{ if } q\in \Big[\max\big\{1,\frac{2}{p}\big\},\frac{2^\ast}{p}\Big).
\label{w2q_norm}
\eeq
Because $1<p<2^\ast-1$, we can choose $2<t<2^\ast$, $\frac{2n}{n+2}<t'<2$ with $\frac{1}{t}+\frac{1}{t'}=1$ such that $t'$ lies in the range given in \eqref{w2q_norm}. Therefore
\beq
\int_{S_{2\delta}} \theta_k^2\,dx \leq \|\theta_k\|_{L^t(S_{2\delta})}\|\theta_k\|_{L^{t'}(S_{2\delta})} \leq \|\theta_k\|_{L^t(S_{2\delta})}\|\theta_k\|_{W^{2,t'}(\R^n)}=o(1) \mbox{ as } k\to \infty
\label{l2_norm}
\eeq
because of \eqref{ls_norm} and \eqref{w2q_norm}. Define a $C^\infty_0$-function $\xi:\R\to\R$ with support in $[-2\delta,2\delta]$ with $0\leq \xi\leq 1$ and $\xi|_{[-\delta,\delta]}=1$ and use the test-function $\xi(x_1)\theta_k$ in \eqref{eq_zk}. This leads to 
\begin{align}
\int_{S_\delta} |\nabla & \theta_k|^2\,dx \leq \int_{S_{2\delta}} |\nabla \theta_k|^2\xi \,dx \nonumber\\
 &= \int_{S_{2\delta}} \Gamma(x) |u_k|^{p-1}u_k \theta_k\xi- V(x)\theta_k^2\xi - \theta_k \frac{\partial \theta_k}{\partial x_1}\xi'\,dx \label{crucial}\\
 &\leq \|\Gamma\|_\infty \|u_k^p\|_{L^s(S_{2\delta})} \|\theta_k\|_{L^{s'}(S_{2\delta})}+\|V\|_\infty \|\theta_k\|^2_{L^2(S_{2\delta})} + \|\xi'\|_{L^\infty} \| \frac{\partial \theta_k}{\partial x_1}\|_{L^r(S_{2\delta})} \|\theta_k\|_{L^{r'}(S_{2\delta})} \nonumber
 \end{align}
where $\frac{1}{s}+\frac{1}{s'}=1$, $\frac{1}{r}+\frac{1}{r'}=1$. Since $1<p<2^\ast-1$, we may arrange that $ps,s'\in (2,2^\ast)$. Furthermore we can choose $r$ in the range given in \eqref{w2q_norm} and additionally $r'\in (2,2^\ast)$. Estimating $\|\theta_k\|_{L^{s'}(S_{2\delta})}\leq C\|\theta_k\|_{H^1(\R^n)}\leq C\big(\|u_k\|_{H^1(\R^n)}+ \|\zeta_k\|_{H^1(\R^n)}\big)=O(1)$ and $\|\frac{\partial \theta_k}{\partial x_1}\|_{L^r(S_{2\delta})}\leq \|\theta_k\|_{W^{2,r}(\R^n)} = O(1)$ by \eqref{w2q_norm} and using \eqref{ls_norm}, \eqref{l2_norm}, we deduce from \eqref{crucial} that 
$\|\theta_k\|_{H^1(S_\delta)}\to 0$ as $k\to\infty$, which together with $u_k=\theta_k+\zeta_k$ and $\|\zeta_k\|_{H^1(\R^n)}\to 0$ yields $\|u_k\|_{H^1(S_\delta)}\to 0$ as $k\to\infty$ in contradiction to Proposition~\ref{prop_result}. Hence we now know that \eqref{true} holds. 
\end{proof}

\subsection{Proof of Theorem \ref{non_ex}} \label{sec:non_ex}

First we prove $c=c_1$ by showing the two inequalities $c\leq c_1$, $c\geq c_1$. 
The first inequality follows from the next lemma and holds always (independently of the ordering of $V_1, V_2$ and $\Gamma_1, \Gamma_2$ assumed in Theorem \ref{non_ex}).

\smallskip

\blem\label{L:c_less_c12}
Assume (H1)--(H3) and $0<\min \sigma(L)$. Then $c\leq \min\{c_1,c_2\}.$
\elem

\begin{proof}
Let $w_1$ be a ground state for the purely periodic problem with coefficients $V_1$, $\Gamma_1$ and define $u_t(x) := w_1(x-te_1), t\in \N$. Then (setting $\eta = \frac{1}{2}-\frac{1}{p+1}$) we compute
\begin{align*}
\int_{\R^n} |\nabla u_t|^2+V(x) |u_t|^{p+1}\,dx &= \int_{\R^n} \left(|\nabla u_t|^2+V_1(x)u_t^2\right)\,dx + \int_{\R^n_-} \left(V_2(x)-V_1(x)\right)u_t^2 \,dx\\
&= \frac{c_1}{\eta} + o(1)
\end{align*}
and
$$
\int_{\R^n} \Gamma(x) |u_t|^{p+1}\,dx = \int_{\R^n} \Gamma_1(x) |u_t|^{p+1}\,dx + \int_{\R^n_-} \left(\Gamma_2(x)-\Gamma_1(x)\right)|u_t|^{p+1} \,dx = \frac{c_1}{\eta} + o(1)
$$
since the integrals over $\R^n_-$ converge to $0$ as $t\to \infty$. Note that $c_1=\eta\int_{\R^n} (|\nabla u_t|^2+V_1(x)u_t^2)\,dx>0$ because $0<\min\sigma(L_1)$. Thus for large $t$ we find $\int_{\R^n} |\Gamma(x) |u_t|^{p+1}\,dx>0$ and hence we can determine $s\in \R$ such that $ su_t\in N$, i.e.,
$$
s^{p-1} =  \frac{\int_{\R^n} |\nabla u_t|^2+ V(x) u_t^2\,dx}{\int_{\R^n} \Gamma(x) |u_t|^{p+1}\,dx}.
$$
Thus
\begin{align*}
J[su_t] = & \eta s^2 \int_{\R^n} |\nabla u_t|^2+ V(x) u_t^2\,dx = \eta \frac{\left(\int_{\R^n} |\nabla u_t|^2+ V(x) u_t^2\,dx\right)^\frac{p+1}{p-1}}{\left(\int_{\R^n} \Gamma(x) |u_t|^{p+1}\,dx \right )^\frac{2}{p-1}} \\
= & \eta \frac{\left(c_1/\eta+ o(1)\right)^\frac{p+1}{p-1}}{\left( c_1/\eta+o(1)\right )^\frac{2}{p-1}}\to c_1 \mbox{ as } t \to \infty.
\end{align*}
This shows that $c\leq c_1$. Similarly, if $w_2$ is a ground state of the purely periodic problem with coefficients $V_2$, $\Gamma_2$, we can define $u_t(x) := w_2(x+te_1)$ with $t\in \N$. Letting $t$ tend to $\infty$, we can see as above that $c\leq c_2$. 
\end{proof}

\smallskip

\noindent
Next we prove that under the assumptions of Theorem~\ref{non_ex} one has $c_1\leq c$. Let $u\in N$. Then 
$$
\int_{\R^n} \Gamma_1(x) |u|^{p+1}\,dx \geq \int_{\R^n}\Gamma(x) |u|^{p+1} \,dx = \int_{\R^n} |\nabla u|^2+V(x) u^2\,dx >0
$$
and therefore we can determine $\tau\in \R$ such that $ \tau u\in N_1$, i.e.,
\begin{align*}
\tau^{p-1} = & \frac{\int_{\R^n} |\nabla u|^2+ V_1(x) u^2\,dx}{\int_{\R^n} \Gamma_1(x) |u|^{p+1}\,dx} \\
 = & \frac{\int_{\R^n} \Gamma(x)|u|^{p+1} + \bigr(V_1(x)-V(x)\bigl) u^2\,dx}{\int_{\R^n} \Gamma_1(x) |u|^{p+1}\,dx} \\
 \leq & 1 
\end{align*}
since $\Gamma\leq \Gamma_1$ and $V_1\leq V$ in $\R^n$. Therefore 
\begin{align*} 
c_1 \leq J_1[\tau u] =  & \eta \tau^2\int_{\R^n} |\nabla u|^2 + V_1(x) u^2\,dx \\
 \leq & \eta \int_{\R^n} |\nabla u|^2+ V(x) u^2 \,dx \\
 = & J[u] \mbox{ since } u \in N.
\end{align*}
Since $u\in N$ is arbitrary, we see that $c_1\leq c$. Now suppose for contradiction that a minimizer $\bar u_0\in N$ of the functional $J$ exists. Then the value $\tau$ s.t. $\tau\bar u_0 \in N_1$ in the above calculation is strictly less than $1$ since we may assume $\bar u_0>0$ almost everywhere on $\R^n$ (cf. Remark~\ref{positivity} and Lemma A2) and also $\Gamma< \Gamma_1$ and/or $V_1<V$ on a set of nonzero measure. However, $\tau<1$ implies 
$c_1\leq J_1[\tau \bar u_0]<J[\bar u_0]=c$, which contradicts Lemma~\ref{L:c_less_c12}. This shows that no strong ground state of \eqref{int_nls} can exist and the proof of Theorem \ref{non_ex} is thus finished. \hfill\qed

\bigskip

We explain next the statement of Remark~\ref{rem:curved}. The distance of $q_j$ to $\Omega_2$ diverges to $\infty$ as $j\rightarrow \infty$ and we can thus use the same argument as in the proof of Theorem \ref{non_ex} with $u_t(x):=w_1(x-q_j), j\gg 1$ and with $\R_-^n$ replaced by $\Omega_2$, cf. Figure~\ref{F:curved_interf_ray}.
\begin{figure}[!ht]
  \begin{center}
    \scalebox{0.28}{\includegraphics{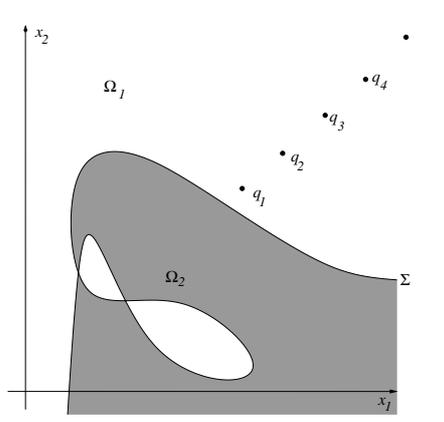}}
  \end{center}
  \caption{An example of the sequence of points $q_j$.}
  \label{F:curved_interf_ray}
\end{figure} 


\subsection{Proof of Theorem \ref{gen_princ}} \label{sec:principle}

Let us first treat the case $c_1\leq c_2$. Similarly to the approach of Arcoya, Cingolani and G\'{a}mez \cite{acg} we consider $u_t(x) := w_1(x-te_1)$ for large $t\in\N$, i.e., we shift the ground state $w_1$ far to the right. Recall from \cite{pankov} that 
\begin{equation}
|w_1(x)| \leq Ce^{-\lambda |x|} \mbox{ for appropriate $C,\lambda>0$}. \label{decay}
\end{equation}
As in the proof of Lemma~\ref{L:c_less_c12} we have $\int_{\R^n} \Gamma(x)|u_t|^{p+1}\,dx=c_1/\eta+o(1)>0$ for large $t\in \N$. Therefore we can determine $s\in \R$ such that $ su_t\in N$, i.e.,
\beq\label{E:s_formula}
s^{p-1} = \frac{\int_{\R^n} |\nabla u_t|^2+ V(x) u_t^2\,dx}{\int_{\R^n} \Gamma(x) |u_t|^{p+1}\,dx}.
\eeq
Next we compute (using again $\eta =\frac{1}{2} -\frac{1}{p+1}$)
\begin{align*}
J[su_t] = & \eta s^2 \int_{\R^n} |\nabla u_t|^2+ V(x) u_t^2\,dx = \eta \frac{\left(\int_{\R^n} |\nabla u_t|^2+ V(x) u_t^2\,dx\right)^\frac{p+1}{p-1}}{\left(\int_{\R^n} \Gamma(x) |u_t|^{p+1}\,dx \right )^\frac{2}{p-1}} \\
= & \eta \frac{\left(\int_{\R^n} |\nabla u_t|^2+ V_1(x) u_t^2\,dx + \int_{\R^n_-} (V_2(x)-V_1(x))u_t^2 \,dx\right)^\frac{p+1}{p-1}}{\left(\int_{\R^n} \Gamma_1(x) |u_t|^{p+1}\,dx+\int_{\R^n_-} (\Gamma_2(x)-\Gamma_1(x))|u_t|^{p+1}\,dx \right )^\frac{2}{p-1}} \\
= & c_1 \frac{1+ \frac{p+1}{p-1} \left[\int_{\R^n} |\nabla u_t|^2+ V_1(x) u_t^2\,dx\right]^{-1} \int_{\R^n_-} (V_2(x)-V_1(x))u_t^2\,dx(1+o(1))}{1+ \frac{2}{p-1}\left[\int_{\R^n} \Gamma_1(x) |u_t|^{p+1}\,dx\right]^{-1}\int_{\R^n_-} (\Gamma_2(x)-\Gamma_1(x))|u_t|^{p+1}\,dx(1+o(1))},
\end{align*}
where $o(1)\rightarrow 0$ as $t\rightarrow \infty$. The last equality holds since both integrals over $\R^n_-$ converge to $0$ as $t\to \infty$ and $(1+\eps)^\alpha = 1+\alpha \eps(1+o(1))$ as $\eps\rightarrow 0$. Hence, we obtain $J[su_t]<c_1$ if 
\beq\label{E:V_Gam_cond}
(p+1) \int_{\R^n_-} (V_2(x)-V_1(x)) w_1(x-te_1)^2\,dx < 2 \int_{\R^n_-}  (\Gamma_2(x)-\Gamma_1(x)) |w_1(x-te_1)|^{p+1}\,dx
\eeq
for sufficiently large $t\in \N$. This establishes condition \eqref{E:V_cond}. Moreover, rewriting \eqref{E:V_Gam_cond} as
$$
\int_{\R^n_-} w_1(x-te_1)^2 \left( (p+1)\delta V(x)- 2\delta\Gamma(x)  |w_1(x-te_1)|^{p-1}\right)\,dx<0
$$
with $\delta V=V_2-V_1$ and $\delta\Gamma=\Gamma_2-\Gamma_1$ and using the decay \eqref{decay} of $w_1$, we see that the above condition is always satisfied for large $t\in \N$ provided $\esssup \delta V< 0$. 

\smallskip

The case $c_2\leq c_1$ is symmetric to that above. One simply needs to shift a ground state $w_2$ to the left. Hence, the proof is the same but with $w_1,V_1,\Gamma_1, c_1, t$ and $\R^n_-$ replaced by $w_2,V_2,\Gamma_2, c_2, -t$, and $\R^n_+$. \hfil\qed


\section{$n$ Dimensions: Examples}\label{S:nD_examples}

Let us state the assumptions on the coefficients once for all the examples below. Namely, we take for $V_1, V_2, \Gamma_1$  and $\Gamma_2$ bounded functions such that (H1)--(H2) are satisfied. The exponent $p$ is assumed to satisfy (H3).

\subsection{Proof of Theorem \ref{T:scaled_interf} (Left and Right States Related by Scaling)}\label{S:ex_scaling}
Due to the specific scaling of $V_1,V_2$ and $\Gamma_1,\Gamma_2$ the ground states $w_1$, $w_2$ of the purely periodic problems \eqref{per_nls} are related as follows: given a ground state $w_2$ the transformation
$$
w_1(x) = \left(\frac{k}{\gamma}\right)^\frac{2}{p-1} w_2(k x)
$$
produces a corresponding ground state $w_1$. Hence, with $\eta = \frac{1}{2}-\frac{1}{p+1}$ we find
\begin{align*}
c_1 = & \eta \int_{\R^n} |\nabla w_1|^2 + V_1(x) w_1^2 \,dx \\
 = & \eta \left(\frac{k}{\gamma}\right)^\frac{4}{p-1}k^2 \int_{\R^n} |\nabla w_2|^2(k x) + V_2(k x) w_2^2(k x) \,dx  \\
 = & \left(\frac{k}{\gamma}\right)^\frac{4}{p-1}k^{2-n} c_2.
 \end{align*}
By assumption $k^\frac{n+2-p(n-2)}{p-1}=k^{\frac{4}{p-1}+2-n}\leq \gamma^\frac{4}{p-1}$. Then $\min\{c_1,c_2\}=c_1$. Now we can achieve \eqref{E:V_cond} in Theorem~\ref{gen_princ} for large $t\in \N$ if we assume $\esssup_{[0,1]^n}(V_2-V_1)<0$. In our case, where $V_1(x) = k^2 V_2(k x)$, this is ensured by assumption \eqref{E:V_scale_ord}. \hfill \qed

\subsection{Proof of Theorem \ref{T:large_jump} (Large Enough Jump in $\Gamma$)}\label{S:large_Gam}

We start this section by a lemma, which explains that the ground-state energy of the periodic problem $L_1u=\Gamma_1|u|^{p-1}u$ depends monotonically on the coefficient $\Gamma_1$ if $\essinf \Gamma_1>0$ and if we keep $V_1$ fixed. To denote the dependence on the coefficient $\Gamma_1$, let us write $c_1(\Gamma_1)$ for the ground-state energy, $w_1(x;\Gamma_1)$ for a ground-state, $N_1(\Gamma_1)$ for the Nehari-manifold, and $J_1[u;\Gamma_1]$ for the energy-functional. 

\begin{lemma}
Assume $\Gamma_1\geq \Gamma_1^*$ with $\essinf\Gamma_1^*>0$. Then
$$
c_1(\Gamma_1) \leq c_1(\Gamma^*_1).
$$
\label{mono}
\end{lemma}

\begin{proof}
Let us select $s$ such that $s w_1(\cdot;\Gamma_1^*)\in N_1(\Gamma_1)$, i.e.,
\[s^{p-1} = \frac{\int_{\R^n} \Gamma_1^*(x)|w_1(x;\Gamma_1^*)|^{p+1} dx}{\int_{\R^n} \Gamma_1(x)|w_1(x;\Gamma_1^*)|^{p+1} dx}.\]
Clearly, $s \leq 1$ and thus we get
\begin{eqnarray*}
c_1(\Gamma_1)&\leq & J_1[s w_1(\cdot;\Gamma_1^*);\Gamma_1]\\
& = & s^2 \eta \int_{\R^n} |\nabla w_1(x;\Gamma_1^*)|^2+ V_1(x) (w_1(x;\Gamma_1^*)^2 dx  \\
& = & s^2 J_1[w_1(\cdot;\Gamma_1^*);\Gamma^*_1] \\
& \leq & J_1[w_1(\cdot;\Gamma_1^*);\Gamma^*_1]=c_1(\Gamma_1^*).
\end{eqnarray*}
\end{proof}

Now we can give the proof of Theorem~\ref{T:large_jump}. By assumption we have $V_1> V_2$ almost everywhere. Once we have checked $c_1(\Gamma_1)\leq c_2(\Gamma_2)$, then we can directly apply Theorem~\ref{gen_princ} to deduce the existence of a strong ground state. Using $\essinf \Gamma_1\geq \beta_0$ and applying Lemma~\ref{mono} with $\Gamma_1^*=\beta_0$, we get 
$$
c_1(\Gamma_1) \leq c_1(\beta_0)=\beta_0^\frac{-2}{p-1}c_1(1).
$$
Hence, by choosing
$$
\beta_0 = \left(\frac{c_2(\Gamma_2)}{c_1(1)}\right)^\frac{1-p}{2},
$$
we obtain $c_1(\Gamma_1) \leq c_2(\Gamma_2)$. This finishes the proof of Theorem~\ref{T:large_jump}. \hfill \qed

\section{One Dimension: General Existence Result (Proof of Theorem \ref{T:Bloch_cond})} \label{S:one-d}

In the case of one dimension we introduce a spectral parameter $\lambda \in \R$ into the problem, i.e., we consider the differential operator
\begin{equation}
L_\lambda := -\frac{d^2}{dx^2} + V-\lambda \mbox{  on  } \cD(L_\lambda)=H^2(\R)\subset L^2(\R)
\label{E:L_1D}
\end{equation}
and look for strong ground states of 
\begin{equation}
L_\lambda u = \Gamma(x) |u|^{p-1}u \mbox{ in } \R.
\label{int_nls_1d}
\end{equation}
The functions $V$ and $\Gamma$ are defined via the bounded periodic functions $V_1, V_2, \Gamma_1, \Gamma_2$ as before.

\medskip

The statement of Theorem \ref{T:Bloch_cond} uses Bloch modes of the linear equation
\begin{equation}
- u'' + \tilde V(x) u = 0 \mbox{ in } \R
\label{linear_hom}
\end{equation}
with a $1$-periodic, bounded function $\tilde V$. We define them next. If we assume that $0<\min\sigma(-\frac{d^2}{dx^2}+ \tilde V(x))$, then \eqref{linear_hom} has two linearly independent solutions (Bloch modes) of the form
$$
u_\pm(x) = p_\pm(x) e^{\mp \kappa x}
$$
for a suitable value $\kappa>0$ and $1$-periodic, positive functions  $p_\pm$. We use the normalization $\|p_\pm\|_\infty =1$. 

\medskip

We summarize next the structure of the proof of Theorem~\ref{T:Bloch_cond}. According to Theorem~\ref{gen_princ}, in the case $c_1\leq c_2$, we have to check the condition
\beq\label{E:1d_cond}
(p+1) \int_{-\infty}^0 \delta V(x) w_1(x-t)^2\,dx < 2 \int_{-\infty}^0  \delta \Gamma(x) |w_1(x-t)|^{p+1}\,dx \quad \text{for} \ t\gg 1,
\eeq
where $\delta V(x)= V_2(x)-V_1(x)$ and $\delta \Gamma(x)= \Gamma_2(x)-\Gamma_1(x)$. We first show in Lemma~\ref{nonlinear}, via a comparison principle, that for $\pm x \rightarrow \infty$ ground states $w_1$ behaves like the Bloch mode $u_\pm$ of \eqref{E:L_1D}. Then in Lemma~\ref{linear_lemma} we compute the asymptotic behavior of the two sides of the inequality \eqref{E:1d_cond} as $t\to \infty$. Since the left hand side behaves like $e^{-2\kappa t}$ whereas the right-hand side behaves like $e^{-(p+1)\kappa t}$, the verification of \eqref{E:1d_cond} may be reduced to 
$$
\int_{-\infty}^0 \delta V(x) w_1(x-t)^2\,dx<0 \mbox{ for } t\gg 1.
$$
In fact, Lemma~\ref{linear_lemma} provides an asymptotic formula for the left hand side of \eqref{E:1d_cond} where $w_1(x-t)$ is replaced by the Bloch mode $u_-(x-t)=p_-(x-t)e^{\kappa(x-t)}$ and, using a geometric series, the integral over the interval $(-\infty,0)$ is replaced by a single period $(-1,0)$. As a result,  \eqref{E:1d_cond} is equivalent to \eqref{gen_cond1}. This finishes the proof of part (a) of Theorem~\ref{T:Bloch_cond}.

In order to prove part (b) of Theorem~\ref{T:Bloch_cond} for $\lambda \ll -1$, we show in Lemma~\ref{lambda_est}-\ref{asymptotic} that $\kappa-\sqrt{|\lambda|}=O(1/\sqrt{|\lambda|})$ and that the periodic part $p_-$ of the Bloch mode $u_-$ converges uniformly to $1$ as $\lambda \rightarrow -\infty$. As a result, for $\lambda \ll -1$ the sign of the integral in \eqref{gen_cond1} is dominated by the local behavior of $V_2(x)-V_1(x)$ near $x=0$ as detailed in Lemma~\ref{L:integ_asympt}.

\medskip

We begin our analysis with the following version of the comparison principle.

\begin{lemma}(Comparison principle) Assume that $\tilde V:\R\to\R$ is bounded, 1-periodic such that $0<\min\sigma(-\frac{d^2}{dx^2}+ \tilde V(x))$. Let $p_\pm e^{\mp\kappa x}$ be the Bloch modes for the operator $-\frac{d^2}{dx^2}+ \tilde V(x)$ and set $P_\pm := \frac{\sup_{[0,1]} p_\pm}{\inf_{[0,1]} p_\pm}$.
\begin{itemize}
\item[(a)] Let $u>0$, $u\in H^1(\R)$ solve 
$$
- u'' + \tilde V(x) u \leq 0 \mbox{ for } |x|\geq |x_0|>0
$$
for some fixed $x_0 \in \R$. If we set $P:=\max\{P_+, P_-\}$, then  
$$
0 < u(x) \leq P e^{-\kappa(|x|-|x_0|)} \max_I u \quad\mbox{ for all } x \in \R,
$$
where $I=[-|x_0|,|x_0|]$.
\item[(b)] Let $u>0$, $u\in H^1(\R)$ solve
$$
- u'' + \tilde V(x) u \geq 0 \mbox{ for } |x|\geq |x_0|>0
$$
for some fixed $x_0 \in \R$. If we set $Q:=\min\{\frac{1}{P_+}, \frac{1}{P_-}\}$, then  
$$
u(x) \geq Q e^{-\kappa(|x|-|x_0|)} \min_I u \quad\mbox{ for all } x \in \R.
$$
\end{itemize}
\label{comparison}
\end{lemma} 

\begin{proof} The proof is elementary and may be well known. We give the details for the reader's convenience. Let us concentrate on the case (a) and the estimate on the interval $[x_0,\infty)$ and suppose that $x_0\geq 0$. The estimate for the interval $(-\infty,-x_0]$ is similar. Due to the assumption $0<\min\sigma(-\frac{d^2}{dx^2}+ \tilde V(x))$ we have the positivity of the quadratic form, i.e.,
\begin{equation}
\int_\R {\phi'}^2+ \tilde V(x) \phi^2 \,dx \geq 0 \mbox{ for all } \phi \in H^1(\R).
\label{pos_quad}
\end{equation}
Let $\psi := u-su_+$ with $u_+(x) = p_+(x)e^{-\kappa x}$ being a Bloch mode satisfying 
$$
\Big(-\frac{d^2}{dx^2}+ \tilde V(x)\Big)u_+=0 \mbox{ on } \R,
$$
and choose 
$$
s := \frac{e^{\kappa x_0}\max_I u }{\inf_{[0,1]} p_+}
$$
so that $\psi(x_0)\leq 0$. Since $\psi$ satisfies
$$
-\psi'' + \tilde V(x) \psi \leq 0 \mbox{ on } (x_0,\infty),
$$
testing with $\psi^+ := \max\{\psi,0\}$ yields
$$
\int_{x_0}^\infty ({\psi^+}')^2+ \tilde V(x)(\psi^+)^2 \,dx = \int_{x_0}^\infty \psi' {\psi^+}'+ \tilde V(x) \psi \psi^+\,dx, \leq 0
$$
which, after extending $\psi^+$ by $0$ to all of $\R$, together with the positivity of the quadratic form in \eqref{pos_quad} yields $\psi^+ \equiv 0$. This implies the claim in case (a). In case (b) on the interval $[x_0,\infty)$ one considers the function $\psi=u-su_+$ with 
$$
s := \frac{e^{\kappa x_0}\min_I u }{\sup_{[0,1]} p_+}
$$ 
and shows that $\psi^-=\max\{-\psi,0\}\equiv 0$ similarly as above.
\end{proof} 

For the next result note that the Wronskian
$$
\det\left(\begin{array}{ll}
u_+ & u_- \\
u_+' & u_-'
\end{array}
\right) =: \omega,
$$ 
constructed from the linearly independent Bloch modes $u_\pm$ of \eqref{linear_hom},
is a constant. 

\begin{lemma} Assume that $\tilde V:\R\to\R$ is bounded, 1-periodic such that $0<\min\sigma(-\frac{d^2}{dx^2}+ \tilde V(x))$ and let $p_\pm e^{\mp\kappa x}$ be the Bloch modes for the operator $-\frac{d^2}{dx^2}+ \tilde V(x)$. If $f(x) = O(e^{-\alpha |x|})$ for $|x|\to \infty$ with $\alpha>\kappa$ and if  $u$ is a solution of 
$$
- u'' + \tilde V(x) u = f(x) \mbox{ on } \R, \quad \lim_{|x|\to \infty} u(x)=0,
$$
then
$$
\lim_{x\to \pm\infty} \frac{u(x)}{u_\pm(x)} = \frac{1}{\omega}\int_{-\infty}^\infty u_\mp(s) f(s)\,ds.
$$
\label{linear_inh}
\end{lemma}

\begin{proof} By the variation of constants formula for inhomogeneous problems we obtain
$$
u(x) = \frac{1}{\omega}\left(\int_{-\infty}^x u_-(s)f(s)\,ds\right)u_+(x)+ \frac{1}{\omega}\left(\int_x^\infty u_+(s)f(s)\,ds\right) u_-(x),
$$
where the boundary condition $\lim_{|x|\to \infty} u(x)=0$ is satisfied because $u_\pm(x)\rightarrow 0$ as $x\rightarrow \pm \infty$ and the integrals are bounded as functions of $x\in \R$ due to the assumption $f(x)= O(e^{-\alpha |x|})$ with $\alpha>\kappa$.
The claim of the lemma follows since again the assumption $\alpha>\kappa$ implies 
$$
\int_x^\infty u_+(s)f(s)\,ds\, \frac{u_-(x)}{u_+(x)} \to 0 \mbox{ as } x \to \infty \mbox{ and } \int_{-\infty}^x u_-(s)f(s)\,ds\, \frac{u_+(x)}{u_-(x)} \to 0 \mbox{ as } x \to -\infty. 
$$
\end{proof}

Now we can describe the behavior of ground states $w_1$ of \eqref{per_nls} for $i=1$ with $V_1$ replaced by $V_1-\lambda$.

\begin{lemma} Assume that $V_1:\R\to\R$ is bounded, 1-periodic and let $\lambda<\min\sigma(-\frac{d^2}{dx^2}+ V_1(x))$. If $\Gamma_1:\R\to\R$ is bounded and if $z>0$ is a solution (not necessarily a ground state) of 
$$
- z'' + (V_1(x)-\lambda) z = \Gamma_1(x) z^p \mbox{ in } \R, \quad \lim_{|x|\to \infty} z(x)=0,
$$
then 
$$
d_\pm(\lambda;z) := \lim_{x \to\pm\infty} \frac{z(x)}{u_\pm(x)} = \frac{1}{\omega}\int_{-\infty}^\infty u_\mp(s)\Gamma_1(s) z(s)^p\,ds,
$$
where $u_\pm$ are the Bloch modes of $-\frac{d^2}{dx^2}+ V_1(x)-\lambda$.
\label{nonlinear}
\end{lemma}

\begin{proof} Let $\epsilon>0$. Then there exists $x_0=x_0(\epsilon)$ such that
$$
-z'' + (V_1(x)-\lambda-\epsilon)z \leq 0 \mbox{ for } |x|\geq |x_0|.
$$
By the comparison principle of Lemma~\ref{comparison} we get the estimate
$$
z(x) \leq \const \|z\|_\infty e^{-\kappa_{\lambda+\epsilon}(|x|-|x_0|)} \mbox{ for all } x \in \R.
$$
Since the map $\lambda\to \kappa_{\lambda}$ is continuous, cf. Allair, Orive \cite{all_or}, we can choose $\epsilon>0$ so small that $p\kappa_{\lambda+\epsilon}>\kappa_\lambda$. Hence the assumptions of Lemma~\ref{linear_inh} with $f(x)=\Gamma_1(x) z(x)^p$ are fulfilled and the claim follows.
\end{proof}

\begin{lemma} The integrals in \eqref{E:1d_cond} have the asymptotic form 
\begin{align*}
\int_{-\infty}^0 \delta V(x)w_1(x-t)^2\,dx &= 
e^{-2\kappa t} \left( \frac{d_-(\lambda;w_1)^2}{1-e^{-2\kappa}} \int_{-1}^0 \delta V(x) p_-(x)^2 e^{2\kappa x}\,dx + o(1)\right), \\
\int_{-\infty}^0 \delta \Gamma(x)|w_1(x-t)|^{p+1}\,dx &= 
e^{-(p+1)\kappa t}\left( \frac{d_-(\lambda;w_1)^{p+1}}{1-e^{-(p+1)\kappa}} \int_{-1}^0 \delta \Gamma(x) p_-(x)^{p+1} e^{(p+1)\kappa x}\,dx + o(1)\right),
\end{align*}
where $o(1)\to 0$ as $t\to\infty, t\in \N$.
\label{linear_lemma}
\end{lemma} 

\noindent
{\bf Remark.} Note that the resulting integrals on the right hand side are independent of the nonlinear ground state. 

\medskip

\begin{proof} For exponents $r\geq 2$ and a $1$-periodic bounded function $q$ let us write 
\begin{eqnarray}
I(t) & = & \int_{-\infty}^0 q(x) w_1(x-t)^r\,dx \nonumber \\
 & = &  d_-(\lambda;w_1)^r\int_{-\infty}^0 q(x) u_-(x-t)^r\,dx + E(t). \label{error}
\end{eqnarray}
For an estimation of the error term $E(t)$ we use Lemma~\ref{nonlinear}. Given $\epsilon>0$, there exists $K=K(\epsilon)>0$ such that 
$$
|w_1(x)^r-d_-(\lambda;w_1)^r u_-(x)^r| \leq \epsilon u_-(x)^r \mbox{ for all } x \leq -K,
$$
i.e.,
$$
|w_1(x-t)^r-d_-(\lambda;w_1)^r u_-(x-t)^r| \leq \epsilon u_-(x-t)^r \mbox{ for all } x \leq 0, t\geq K.
$$
Hence, for $t\geq K$
\begin{align*}
E(t) &\leq \epsilon \|q\|_\infty \int_{-\infty}^0 p_-(x-t)^r e^{r\kappa(x-t)}\,dx \\ 
& \leq \epsilon \|q\|_\infty \underbrace{\|p_-\|^r_\infty}_{=1} \frac{e^{-r\kappa t}}{r\kappa},  
\end{align*}
i.e., $E(t)=e^{-r\kappa t} o(1)$ as $t\rightarrow \infty$.
Next we compute the integral for $t\in \N$
\begin{align*}
\int_{-\infty}^0 q(x) u_-(x-t)^r\,dx &= e^{-r\kappa t} \sum_{k=0}^\infty \int_{-k-1}^{-k} q(x) p_-(x)^r e^{r\kappa x}\,dx \\ 
&= \frac{e^{-r\kappa t}}{1-e^{-r\kappa}} \int_{-1}^0 q(x) p_-(x)^r e^{r\kappa x}\,dx.
\end{align*}
This result shows that for large values of $t$ the dominating part in \eqref{error} is played by the integral w.r.t. the Bloch mode, since in comparison the error term can be made arbitrarily small. This is the claim of the lemma. 
\end{proof}

\medskip

The above computation explains why it is possible to replace the existence condition \eqref{E:1d_cond} by \eqref{gen_cond1}. The reason is that the quadratic term on the left side of \eqref{E:1d_cond} decays like $e^{-2\kappa t}$ whereas the term on the right side decays like $e^{-(p+1)\kappa t}$ as $t\to \infty$. 

\medskip

At this stage note that by Lemma~\ref{linear_lemma} we have proved part (a) of Theorem~\ref{T:Bloch_cond}. It remains to consider part (b), i.e., to decide on the sign of 
\begin{equation}
\int_{-1}^0 \delta V(x) p_-(x)^2 e^{2\kappa x}\,dx
\label{inter_result}
\end{equation}
as $\lambda\to -\infty$. First we investigate the behavior of the linear Bloch modes for $\lambda\to -\infty$. We begin by stating a relation between the spectral parameter and the coefficient of exponential decay of the Bloch modes.

\begin{lemma} Let $\kappa=\kappa(\lambda)$ be the coefficient of exponential decay of the Bloch modes for $-\frac{d^2}{dx^2}+ V_1(x)-\lambda$. Then for $\lambda$ sufficiently negative we have
$$
\frac{-\|V_1\|_\infty}{\kappa+\sqrt{|\lambda|}}\leq \kappa-\sqrt{|\lambda|} \leq \frac{\|V_1\|_\infty}{\kappa+\sqrt{|\lambda|}}.
$$
\label{lambda_est}
\end{lemma}

\begin{proof} We prove the estimate of the difference between $\kappa$ and $\sqrt{|\lambda|}$ via the comparison principle. The Bloch modes $u_\pm(x) = p_\pm(x) e^{\mp \kappa x}$ satisfy
$$
- u_\pm'' + (\|V_1\|_\infty-\lambda) u_\pm \geq 0 \mbox{ and } - u_\pm'' + (-\|V_1\|_\infty-\lambda) u_\pm \leq 0.
$$
Hence, using the comparison principle of Lemma~\ref{comparison}, we get for $\lambda<-\|V_1\|_\infty$
$$
C_1 e^{-\sqrt{\|V_1\|_\infty-\lambda}x} \leq u_+ (x) \leq C_2 e^{-\sqrt{-\|V_1\|_\infty-\lambda}x} \mbox{ for } x \in \R.
$$
This implies
$$
\sqrt{\|V_1\|_\infty-\lambda} \geq \kappa \geq \sqrt{-\|V_1\|_\infty-\lambda}
$$
from which the statement easily follows.
\end{proof}

Next we give a representation of the periodic part of the Bloch mode $u_-(x) = p_-(x)e^{\kappa x}$. 

\begin{lemma} The periodic part $p_-$ of the Bloch mode $u_-$ of the operator $-\frac{d^2}{dx^2}+ V_1(x)-\lambda$ satisfies the differential equation
$$
p_-'' + 2\kappa p_-' +(\kappa^2+\lambda)p_-= V_1(x) p_- \mbox{ on } [-1,0]
$$
with periodic boundary conditions and therefore 
\begin{align*}
p_-(x) =& \left( \frac{\int_{-1}^0 e^{(\kappa-\sqrt{|\lambda|})s}p_-(s) V_1(s) \,ds}{2\sqrt{|\lambda|}(e^{\kappa-\sqrt{|\lambda|}}-1)} +\frac{1}{2\sqrt{|\lambda|}}\int_{-1}^x e^{(\kappa-\sqrt{|\lambda|})s}p_-(s) V_1(s)\, ds\right)e^{(-\kappa+\sqrt{|\lambda|})x} \\
&+ \left( \frac{-\int_{-1}^0 e^{(\kappa+\sqrt{|\lambda|})s}p_-(s) V_1(s) \,ds}{2\sqrt{|\lambda|}(e^{\kappa+\sqrt{|\lambda|}}-1)} -\frac{1}{2\sqrt{|\lambda|}}\int_{-1}^x e^{(\kappa+\sqrt{|\lambda|})s}p_-(s) V_1(s)\, ds\right)e^{(-\kappa-\sqrt{|\lambda|})x}.
\end{align*}
\label{p_rep}
\end{lemma}

\begin{proof} Starting from the solutions $e^{(-\kappa+\sqrt{|\lambda|})x}$, $e^{(-\kappa-\sqrt{|\lambda|})x}$ of the homogeneous equation $p''+2\kappa p'+(\kappa^2+\lambda)p=0$, we get via the variation of constants 
\begin{align*}
p_-(x) =& \left( \alpha+\frac{1}{2\sqrt{|\lambda|}}\int_{-1}^x e^{(\kappa-\sqrt{|\lambda|})s}p_-(s) V_1(s)\, ds\right)e^{(-\kappa+\sqrt{|\lambda|})x} \\
&+ \left(\beta-\frac{1}{2\sqrt{|\lambda|}}\int_{-1}^x e^{(\kappa+\sqrt{|\lambda|})s}p_-(s) V_1(s)\, ds\right)e^{(-\kappa-\sqrt{|\lambda|})x}.
\end{align*}
Inserting the periodicity conditions $p_-(0)=p_-(-1)$ and $p_-'(0)=p_-'(-1)$, we obtain the claim.
\end{proof}

Now we can state the asymptotic behavior of $p_-$ as $\lambda\to -\infty$. 

\begin{lemma} As $\lambda\to -\infty$, we have that $p_-\to 1$ uniformly on $[-1,0]$. More precisely
$$
\|p_- - 1\|_\infty = O\left(\tfrac{1}{\sqrt{|\lambda|}}\right) \quad  \text{as} \ \lambda \rightarrow -\infty
$$
and, in addition,
$$
\sqrt{|\lambda|}(\kappa-\sqrt{|\lambda|}) - \frac{1}{2}\int_{-1}^0 V_1(s)\,ds=  O\left(\tfrac{1}{\sqrt{|\lambda|}}\right) \quad  \text{as} \ \lambda \rightarrow -\infty.
$$
\label{asymptotic}
\end{lemma}

\begin{proof} First one checks by a direct computation using Lemma~\ref{p_rep} that $\|p'_-\|_\infty = O(1/\sqrt{|\lambda|})$ as $\lambda\to -\infty$. Hence, by the normalization $\|p_-\|_\infty=1$ and by continuity of $p_-$ we have $p_-(\xi_\lambda)=1$ for some $\xi_\lambda \in [-1,0]$. By the mean value theorem  
$$
|p_-(x)-1|\leq \|p_-'\|_\infty|x-\xi_\lambda| =O(1/\sqrt{|\lambda|})
$$
as $\lambda\to -\infty$ uniformly for $x\in [-1,0]$. This proves the first part of the lemma. For the second part one first finds again by direct estimates from Lemma~\ref{lambda_est} and Lemma~\ref{p_rep}
\begin{align}
p_-(x) &= \frac{\int_{-1}^0 e^{(\kappa-\sqrt{|\lambda|})s}p_-(s) V_1(s) \,ds}{2\sqrt{|\lambda|}(e^{\kappa-\sqrt{|\lambda|}}-1)} e^{(-\kappa+\sqrt{|\lambda|})x}+ O\left(\tfrac{1}{\sqrt{|\lambda|}}\right) \label{p_l} \\
&= \frac{\int_{-1}^0 p_-(s) V_1(s) \,ds}{2\sqrt{|\lambda|}(\kappa-\sqrt{|\lambda|})}+ O\left(\tfrac{1}{\sqrt{|\lambda|}}\right), \nonumber
\end{align}
where we have used that $e^{(\kappa-\sqrt{|\lambda|})(s-x)}=1+O(1/\sqrt{|\lambda|})$ uniformly in $s,x\in[-1,0]$.
Next we observe by Lemma~\ref{lambda_est} that $\sqrt{|\lambda|}(\kappa-\sqrt{|\lambda|})$ is bounded in absolute value by $\|V_1\|_\infty$, and hence has accumulation points as $\lambda\to-\infty$. Every accumulation point $d$ satisfies 
$$
1 = \frac{1}{2d} \int_{-1}^0 V_1(s) \,ds.
$$
If we set $\gamma_\lambda:= \frac{\int_{-1}^0 p_-(s) V_1(s) \,ds}{2\sqrt{|\lambda|}(\kappa-\sqrt{|\lambda|})}$, then 
$\gamma_\lambda-1 = O(1/\sqrt{|\lambda|})$ by \eqref{p_l} and thus
$$
\sqrt{|\lambda|}(\kappa-\sqrt{|\lambda|}) = \tfrac{1}{2}\int_{-1}^0V_1(s)ds +O(1/\sqrt{|\lambda|}),
$$
which is the second claim of the lemma.
\end{proof}

Finally, we can determine the behavior of the integral in \eqref{inter_result}. 

\begin{lemma} Let $V_1, V_2$ be bounded and 1-periodic and $\delta V:= V_2-V_1$. If there exists $\eps>0$ such that $\delta V$ is continuous and negative in $[-\eps,0)$, then for sufficiently negative $\lambda$
$$
\int_{-1}^0 \delta V(x)p_-(x)^2 e^{2\kappa x}\,dx < 0.
$$
\label{L:integ_asympt}
\end{lemma}

\begin{proof}
Since $\delta V$ is continuous and negative on $[-\eps,0)$, there exist $\alpha>0$ such that $\delta V(x)<-\alpha$ for $x\in [-\eps, -\eps/2]$. Using that $p_-(x)^2-1=O(1/\sqrt{|\lambda)|})$, we estimate
\begin{eqnarray*}
\lefteqn{\int_{-1}^0 \delta V(x) p_-(x)^2 e^{2\kappa x}\,dx} \\
&\leq & \int_{-1}^{-\eps/2} \delta V(x) e^{2\kappa x}\,dx + \int_{-1}^{-\eps/2} \delta V(x)\big(p_-(x)^2-1\big)e^{2\kappa x}\,dx \\
& \leq & \int_{-1}^{-\eps} \delta V(x) e^{2\kappa x}\,dx -\alpha \int_{-\eps}^{-\eps/2} e^{2\kappa x}\,dx + \|\delta V\|_\infty O\Big(\frac{1}{\sqrt{|\lambda|}}\Big)\int_{-1}^{-\eps/2} e^{2\kappa x}\,dx  \\
& \leq & \frac{\|\delta V\|_\infty}{2\kappa}e^{-2\kappa \eps} -
\frac{\alpha}{2\kappa}(e^{-\kappa \eps}-e^{-2\kappa\eps}) + \|\delta V\|_\infty O\Big(\frac{1}{\sqrt{|\lambda|}}\Big) \frac{e^{-\kappa\eps}}{2\kappa} \\
&= & \frac{e^{-2\kappa \eps}}{2\kappa}\left(\|\delta V\|_\infty+\alpha-\alpha e^{\kappa\eps}+\|\delta V\|_\infty O\Big(\frac{1}{\sqrt{|\lambda|}}\Big)e^{\kappa\eps}\right)
\end{eqnarray*}
which is negative for $\lambda\ll -1$ because $\kappa\to \infty$ as $\lambda\to -\infty$. This proves the claim.
\end{proof}

With this the proof of Theorem~\ref{T:Bloch_cond} is complete.


\section{One Dimension: Examples and Heuristics} \label{S:1D_example}

\subsection{A Dislocation Interface Example}

As a particular example of a one-dimensional interface we consider so-called dislocated potentials, i.e., if $V_0, \Gamma_0$ are bounded 1-periodic functions, then we set 
$$
V(x) = \left\{
\begin{array}{ll}
V_0(x+\tau), & x>0, \vspace{\jot}\\
V_0(x-\tau), & x<0,
\end{array} \right.
\qquad
\Gamma(x) = \left\{
\begin{array}{ll}
\Gamma_0(x+\tau), & x>0, \vspace{\jot}\\
\Gamma_0(x-\tau), & x<0,
\end{array} \right.
$$
where $\tau\in\R$ is the dislocation parameter. We consider problem \eqref{int_nls_1d}  and analogously to the notation in \eqref{per_nls} we define $L_{\lambda,1}:= -\frac{d^2}{dx^2}+ V_0(x+\tau)-\lambda$, $L_{\lambda,2}:= -\frac{d^2}{dx^2}+ V_0(x-\tau)-\lambda$, $\Gamma_1(x) := \Gamma_0(x+\tau)$, and $\Gamma_2(x):=\Gamma_0(x-\tau)$.
Note that $c_1=c_2$ in this case. The following is then a direct corollary of Theorem~\ref{T:Bloch_cond} and Remark~\ref{R:bloch_cond_reverse}.

\begin{corollary} Assume that $V_0, \Gamma_0$ are bounded 1-periodic functions on the real line with $\esssup\Gamma_0>0$. Assume moreover that $0< \min \sigma(L_\lambda)$ and $1<p<\infty$.
\begin{itemize}
\item[(a)] Problem \eqref{int_nls_1d} in the dislocation case with $\tau\in \R$ has a strong ground state provided 
\begin{equation}
\int_{-1}^0 \Big(V_0(x-\tau)-V_0(x+\tau)\Big) p_-(x+\tau)^2 e^{2\kappa x}\,dx<0
\label{dis_cond1}
\end{equation}
or 
\begin{equation}
\int_0^1 \Big(V_0(x+\tau)-V_0(x-\tau)\Big) p_+(x-\tau)^2 e^{-2\kappa x}\,dx<0,
\label{dis_cond1'}
\end{equation}
where $p_\pm(x) e^{\mp\kappa x}$ with $\kappa>0$ are the Bloch modes of the operator $-\frac{d^2}{dx^2}+ V_0(x)-\lambda$. 
\item[(b)] For $\lambda<0$ sufficiently negative at least one of the conditions \eqref{dis_cond1}, \eqref{dis_cond1'} is fulfilled provided $V_0$ is a $C^1$-potential and
\begin{equation}
V_0(-\tau)\not = V_0(\tau) \qquad \mbox{ or } \qquad V_0(-\tau)= V_0(\tau) \mbox{ and } V_0'(-\tau) > V_0'(\tau).
\label{dis_cond2}
\end{equation}
For $|\tau|$ sufficiently small the above condition \eqref{dis_cond2} on $V_0$ is fulfilled if
\begin{equation}
\,V_0'(0)\not= 0 \qquad \mbox{ or } \qquad V_0'(0)=0 \mbox{ and } \sign\tau\,V_0''(0)<0,
\label{dis_cond3}
\end{equation}
where for the second part of the condition one needs to assume that $V_0$ is twice differentiable at $0$.
\end{itemize}
\label{dislocate}
\end{corollary}

\noindent
{\bf Remarks.} (1) The case where $V_0'(0)=0$ and $\sign\tau V_0''(0)>0$ is not covered by the above theorem. We believe that for this case strong ground states do not exist. \\
(2) One can also consider the dislocation problem with two parameters, i.e, 
$$
V(x) = \left\{
\begin{array}{ll}
V_0(x+\tau), & x>0, \vspace{\jot}\\
V_0(x-\tau), & x<0,
\end{array} \right.
\qquad
\Gamma(x) = \left\{
\begin{array}{ll}
\Gamma_0(x+\sigma), & x>0, \vspace{\jot}\\
\Gamma_0(x-\sigma), & x<0
\end{array} \right.
$$
where $\tau,\sigma\in\R$ are the dislocation parameters. If $V_0$, $\Gamma_0$ are even, bounded 1-periodic functions and if 
$w_1$ is a ground state for $L_{\lambda,1} w_1 = \Gamma_1(x) |w_1|^{p-1}w_1$ in $\R$, then $w_2(x):=w_1(-x)$ is a ground state for the problem 
$L_{\lambda,2} w_2 = \Gamma_2(x) |w_2|^{p-1}w_2$ in $\R$. One then easily sees that again we have $c_1=c_2$\footnote{Note that the $\lambda$-dependence of $c_1=c_2$ has been dropped.}. The result of Corollary~\ref{dislocate} (a) immediately applies. For the result of Corollary~\ref{dislocate} (b) one only needs to take the second parts of \eqref{dis_cond2}, \eqref{dis_cond3} into account.

\medskip

\noindent
{\em Proof of Corollary \ref{dislocate}:} In the dislocation case the unperturbed energy levels of ground states satisfy $c_1=c_2$ and thus both versions of Theorem~\ref{T:Bloch_cond} are available. If $u_-(x)= p_-(x)e^{\kappa x}$ is the Bloch mode decaying at $-\infty$ of $-\frac{d^2}{dx^2}+ V_0(x)-\lambda$, then $p_-(x+\tau)e^{\kappa (x+\tau)}$ is the corresponding Bloch mode of the operator $L_{\lambda,1}$. Therefore, condition \eqref{gen_cond1} of Theorem~\ref{T:Bloch_cond} amounts to 
$$
e^{2\kappa\tau}\int_{-1}^0 \Big(V_0(x-\tau)-V_0(x+\tau)\Big) p_-(x+\tau)^2 e^{2\kappa x}\,dx<0,
$$
which is equivalent to \eqref{dis_cond1} of the Corollary~\ref{dislocate}. Likewise \eqref{gen_cond2} amounts to $V_0(-\tau)<V_0(\tau)$ or $V_0(-\tau)=V_0(\tau)$ and $V_0'(-\tau)>V_0'(\tau)$. If we apply the version of Theorem~\ref{T:Bloch_cond} given in Remark~\ref{R:bloch_cond_reverse}, then we get \eqref{dis_cond1'} instead of \eqref{dis_cond1} and $V_0(\tau)<V_0(-\tau)$ or $V_0(-\tau)=V_0(\tau)$ and $V_0'(-\tau)>V_0'(\tau)$. This explains \eqref{dis_cond1}, \eqref{dis_cond1'} as well as \eqref{dis_cond2}. The final condition \eqref{dis_cond3} follows via Taylor-expansion
$$
V_0(\tau)-V_0(-\tau)= 2\tau V_0'(0)+o(\tau), \quad V_0'(\tau)-V_0'(-\tau)= 2\tau V_0''(0)+o(\tau)
$$
from \eqref{dis_cond2}. \qed

\subsection{A Heuristic Explanation of Theorem~\ref{T:Bloch_cond}(b) and Corollary~\ref{dislocate}(b) for $\lambda$ Sufficiently Negative}

We provide next a heuristic explanation of the existence results in the 1D interface problem for $\lambda \ll -1$ in Theorem~\ref{T:Bloch_cond}(b) and thus Corollary~\ref{dislocate} (b). In the following we show how to quickly find a function in $N$ with energy smaller than $c_1 (\leq c_2)$ so that the criterion of Theorem~\ref{main} for the existence of ground states is satisfied.

We restrict the discussion to the case $\Gamma_1\equiv \Gamma_2$. The heuristic part of the analysis is the use of the `common wisdom'\footnote{In \cite{KKP2010} this common wisdom has been proved under similar, but not identical assumptions.} that as $\lambda\to -\infty$ each ground state of the purely periodic problem 
$$
- u '' +V_1(x) u = \Gamma_1(x)|u|^{p-1}u \mbox{ in } \R
$$
is highly localized and concentrates near a point $x_0(\lambda)$. We assume $x_0(\lambda)\to x_0^\ast\in (0,1]$ as $\lambda\to -\infty$ and that $x_0^\ast$ is a point, where $V_1$ assumes its minimum. Moreover, we assume below that even at a small distance (e.g. one half period of $V_1$) from the concentration point the ground state decays exponentially fast with increasing distance from $x_0^*$.

Consider a ground state $w_1(x;\lambda)$. The function $sw_1$ lies in $N$ if
$$s^{p-1} = \frac{\int_{-\infty}^\infty {w_1'}^2 +(V(x)-\lambda)w_1^2 dx}{\int_{-\infty}^\infty \Gamma(x) |w_1|^{p+1} dx}.$$
Because $\Gamma_1\equiv \Gamma_2$ and $w_1\in N_1$, the denominator equals $\int_{-\infty}^\infty\Gamma_1(x)|w_1|^{p+1} dx = \int_{-\infty}^\infty {w_1'}^2 +(V_1(x)-\lambda)w_1^2 dx$. Therefore
$$s^{p-1} = \frac{\int_{0}^\infty {w_1'}^2 +(V_1(x)-\lambda)w_1^2 dx + \int_{-\infty}^0 {w_1'}^2 +(V_2(x)-\lambda)w_1^2 dx}{\int_{-\infty}^\infty {w_1'}^2 +(V_1(x)-\lambda)w_1^2 dx }.$$
Due to the concentration of $w_1$ near $x_0^*>0$ as $\lambda \rightarrow -\infty$ and its fast decay as $|x-x_0^*|$ grows, we see that $s<1$ if $V_2(x)<V_1(x)$ in a left neighborhood of $0$.

Finally,
\[
\begin{split}
J[sw_1] & = \int_{-\infty}^\infty \tfrac{s^2}{2} \left({w_1'}^2 +(V(x)-\lambda)w_1^2 \right) -\tfrac{s^{p+1}}{p+1}\Gamma_1(x)|w_1|^{p+1} dx\\ 
& = s^2\left(\tfrac{1}{2}-\tfrac{1}{p+1}\right) \int_{-\infty}^\infty {w_1'}^2 +(V(x)-\lambda)w_1^2 dx\\
& = s^2\left[J_1[w_1] +  \left(\tfrac{1}{2}-\tfrac{1}{p+1}\right)\int_{-\infty}^0 (V_2(x)-V_1(x))w_1^2 dx\right],
\end{split}
\]
where the second equality follows from $sw_1\in N$. We get $J[sw_1]<J_1[w_1]=c_1$ if $V_2(x)<V_1(x)$ in a left neighborhood of $0$, see Figure~\ref{F:heurist_gen} for a sketch. Note that this calculation does not, however, imply that the function $sw_1$ is a ground state of the interface problem \eqref{int_nls_1d} (with $\Gamma_1\equiv \Gamma_2$).
\begin{figure}[!ht]
  \begin{center}
  \scalebox{0.6}{\includegraphics{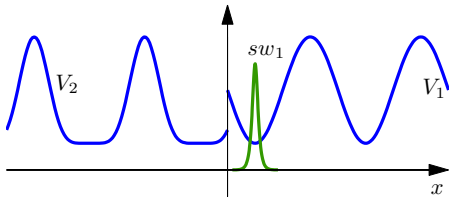}}
  \end{center}
  \caption{An example of a function in $N$ attaining smaller energy than $c_1 (\leq c_2)$ for the 1D interface problem \eqref{int_nls_1d}. Heuristically, the ground state existence conditions of Theorem~\ref{main} are thus satisfied.}
  \label{F:heurist_gen}
\end{figure} 

In the dislocation case, i.e., Corollary~\ref{dislocate} (b), where $c_1=c_2$, the discussion applies analogously if we restrict attention to the case $\Gamma\equiv\Gamma_0\equiv\const$ We denote again $x_0^*=\lim_{\lambda\to -\infty}x_0(\lambda)$ as the point of concentration of a ground state of the purely periodic problem, which now may be any point (left or right of zero) where $V_0$ attains its global minimum. Thus $J[sw_1]<c_1=c_2$ is satisfied if $V_0(x-\tau)<V_0(x+\tau)$ for $x$ in a left neighborhood of $0$ and if we take $x_0^*>0$. Likewise, in the case where $V_0(x-\tau)>V_0(x+\tau)$ for $x$ in a right neighborhood of $0$ we may take $x_0^*<0$. The three possible scenarios, namely $V_0'(0)\neq 0, V_0'(0)=0$ and $V_0$ having a local minimum at $x=0$, and finally $V_0'(0)=0$ and $V_0$ having a local maximum at $x=0$ are depicted schematically in Figure~\ref{F:heurist_disloc}. The full green lines in the columns $\tau>0$ and $\tau<0$ depict functions $sw_1\in N$ with energy smaller than $c_1=c_2$. As the above calculations show, the candidate positioned the closest to $x=0$ produces the smallest energy and has the smallest $s$. It is therefore plotted with the smallest amplitude.
\begin{figure}[!ht]
  \begin{center}
  \scalebox{0.6}{\includegraphics{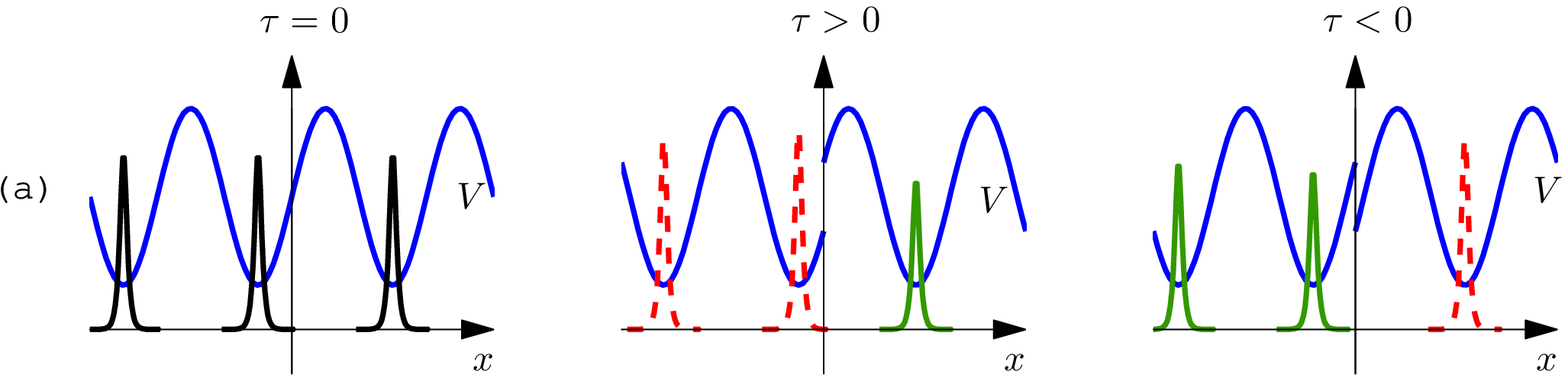}}
  \scalebox{0.6}{\includegraphics{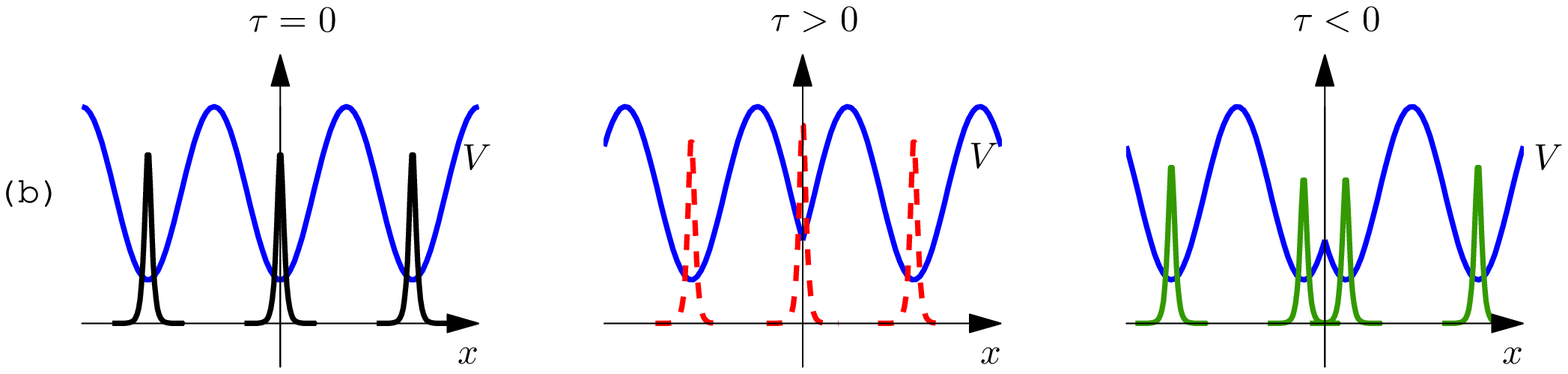}}
  \scalebox{0.6}{\includegraphics{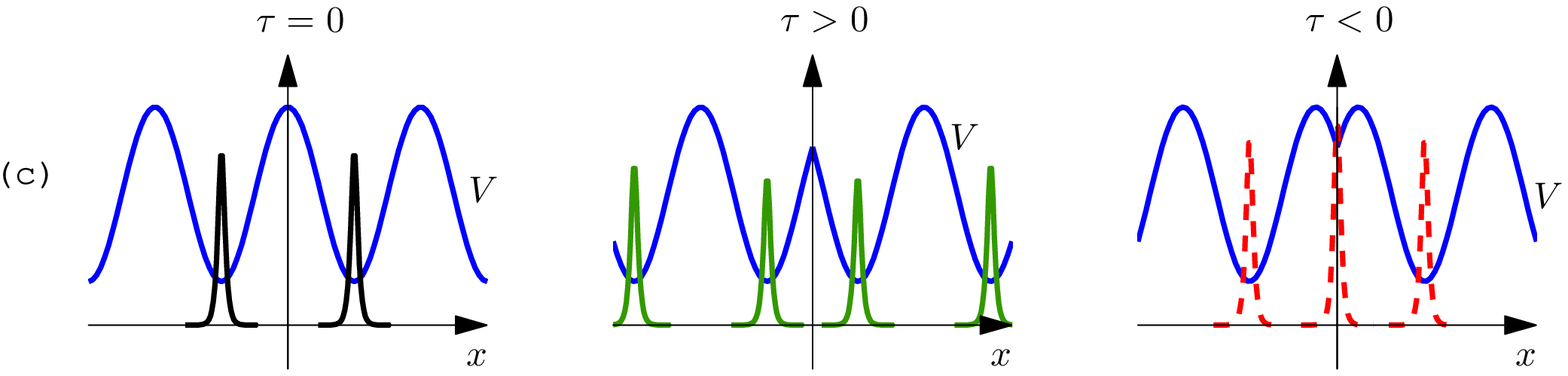}}
  \end{center}
  \caption{Cases (a), (b) and (c) correspond to $V_0\in C^1$ with $V_0'(0)\neq 0$, $V_0$ having a local minimum and $V_0$ having a local maximum at $x=0$, respectively. In the column $\tau=0$ the ground states are plotted in black. In the columns $\tau\neq 0$ the green full and red dashed curves are functions $sw_1 \in N$ with energy smaller and larger than $c_1=c_2$ respectively.}
  \label{F:heurist_disloc}
\end{figure} 

Finally, we mention that in the dislocation case with two dislocation parameters $\tau, \sigma$ (cf. Remark 2 after Corollary~\ref{dislocate}) the above considerations apply if $V_0, \Gamma_0$ are even potentials and if we set $\sigma=0$ so that $\Gamma_1(x)=\Gamma_0(x)=\Gamma_2(x)$. In this setting only cases (b) and (c) of Figure~\ref{F:heurist_disloc} apply.

\section{Discussion, Open Problems}\label{S:discussion}

The above analysis describes the existence of ground state surface gap solitons of \eqref{E:NLS} in the case of two materials meeting at the interface described by the hyperplane $x_1=0$. It would be of interest to generalize this analysis to curved interfaces as well as to several intersecting interfaces with more than two materials.

Besides looking for strong ground states, i.e. global minimizers of the energy $J$ within the Nehari manifold $N$, one can also pose the question of existence of bound states, i.e. general critical points of $J$ in $H^1(\R^n)$, including possibly ground states, i.e. minimizers of $J$ within the set of nontrivial $H^1$ solutions. The existence of such more general solutions is not covered by our results. It is also unclear whether the ground states found in the current paper are unique (up to multiplication with $-1$ and translation by $T_k e_k$ for $k=2,\ldots,n$) and what their qualitative properties are. In particular, it would be interesting to determine the location where the above ground states are `concentrated'. Although their existence is shown using candidate functions shifted along the $x_1$-axis to $+\infty$ or $-\infty$, we conjecture that the ground states are concentrated near the interface at $x_1=0$.

The conditions of our non-existence result in Theorem~\ref{non_ex} and Remark~\ref{rem:curved} agree with the set-up of several optics experiments as well as numerical computations in the literature if one neglects the fact that the periodic structures used in these are finite. In \cite{Szameit07} the authors consider an interface of a homogeneous dielectric medium with $\eps_r=\alpha>0$ and a photonic crystal with $\eps_r=\alpha+Q(x), Q>0$ and 
provide numerics and experiments for surface gap solitons (SGSs) in the semi-infinite gap of the spectrum. The nonlinearity is cubic ($p=3$) and $\Gamma\equiv \text{const.}$ The observed SGSs cannot be modelled as strong ground states of \eqref{int_nls} due to the ordering of $V_1, V_2$ and $\Gamma_1,\Gamma_2$, which enables the application of Theorem~\ref{non_ex}. They could, possibly, be strong ground states if the structure was modelled as a finite block of a photonic crystal with $\eps_r=\alpha+Q(x), Q>0$ embedded in an infinite dielectric with $\eps_r=\alpha$ but this situation is outside the reach of our model. 

A similar situation arises in \cite{Suntsov08}, which studies the interface of two cubically nonlinear photonic crystals with $\Gamma_1\equiv \Gamma_2$ and either $V_1\geq V_2$ or $V_1\leq V_2$ with a strict inequality on a set of nonzero measure. Likewise, in the computations of \cite{do_pel,Blank_Dohnal}, where \eqref{int_nls} is considered in 1D with $V_1\equiv V_2$ and $\Gamma_1,\Gamma_2\equiv \text{const.}$, $\Gamma_1 \neq \Gamma_2$, the SGSs computed in the semi-infinite spectral gap cannot correspond to strong ground states of \eqref{int_nls}.
The findings of \cite{Szameit07,Suntsov08,do_pel,Blank_Dohnal} show that despite the absence of strong ground states bound states may still exist.


\section*{Appendix}

\noindent
{\bf Lemma A1.} {\em Let $V\in L^\infty$ and $L=-\Delta + V(x)$ such that $0\not\in \sigma(L)$. If $1<q<\infty$ then $L^{-1}: L^q(\R^n)\to W^{2,q}(\R^n)$ is a bounded linear operator.}

\medskip

\noindent
{\em Proof.} Since the spectrum of $L$ is stable in $L^q(\R^n)$ with respect to $q\in [1,\infty]$, cf. Hempel, Voigt \cite{he_vo}, we have for all $u \in {\mathcal D}(L)\subset L^q(\R^n)$ that $\|u\|_{L^q(\R^n)} \leq C \|Lu\|_{L^q(\R^n)}$. We need to check that ${\mathcal D}(L)=W^{2,q}(\R^n)$. Here we restrict to $1<q<\infty$. Note that since $V\in L^\infty$, for $u\in {\mathcal D}(L)$ we have that $Lu$, $Vu$ and hence $-\Delta u$ all belong to $L^q(\R^n)$. Therefore
$$
-\Delta u + u = (1-V)u + Lu
$$
and with the help of the Green function $G_{-\Delta +1}$ we find
$$
 u = G_{-\Delta +1}\ast\bigl((1-V)u + Lu¸\bigr).
$$
By the mapping properties of the Green function (cf. mapping properties of Bessel potentials, Stein \cite{stein}), it follows that $u\in W^{2,q}(\R^n)$ and
$$
\|u\|_{W^{2,q}(\R^n)}\leq C\bigl( \|(1-V)u\|_{L^q(\R^n)}+ \|Lu\|_{L^q(\R^n)}\bigr) \leq \bar C \|Lu\|_{L^q(\R^n)}.
$$
Since trivially $W^{2,q}(\R^n)\subset {\mathcal D}(L)$, the proof is done. \qed

\medskip

\noindent
{\bf Lemma A2.} {\em Let $V,\Gamma\in L^\infty$ and let $L=-\Delta + V(x)$ be such that $0<\min\sigma(L)$. If $1<p<2^\ast-1$, then every strong ground state $u_0$ of \eqref{E:NLS} is either positive in $\R^n$ or negative in $\R^n$.}

\medskip

\noindent
{\em Sketch of proof.} Let $u_0$ be a strong ground state. Then $\bar u_0 := |u_0|$ is also a strong ground state and $\bar u_0\not \equiv 0$. Due to the subcritical growth of the nonlinearity we have that locally $\bar u_0$ is a $C^{1,\alpha}$-function. If we define $Z=\{x\in \R^n: \bar u_0(x)=0\}$, then $Z^c=\R^n\setminus Z$ is open and $\nabla\bar u_0\equiv 0$ on $Z$. If we assume that $Z$ is nonempty, then there exists an open ball $B\subset Z^c$ such that $\partial B\cap Z\not = \emptyset$. This contradicts Hopf's maximum principle. Thus $\bar u_0>0$ in $\R^n$ and therefore either $u_0>0$ in $\R^n$ or $u_0<0$ in $\R^n$.
\qed

\bigskip

\noindent
{\bf Acknowledgement.} We thank Andrzej Szulkin (Stockholm University) for discussions leading to the improved hypothesis (H2) instead of (H2').



\end{document}